%% file: main.tex
\documentclass[11pt,a4paper]{amsart}

\usepackage[margin=1in]{geometry}
\usepackage{amsmath,amssymb,amsthm,mathtools}
\usepackage{tikz}
\usetikzlibrary{decorations.markings}
\usepackage{subcaption}
\usepackage{hyperref}
\usepackage{float}
\usepackage{fix-cm}

\newtheorem{theorem}{Theorem}
\newtheorem{proposition}[theorem]{Proposition}
\newtheorem{lemma}[theorem]{Lemma}
\newtheorem{corollary}[theorem]{Corollary}
\newtheorem*{mainformula}{Lemma A}
\newtheorem*{mainintermediate}{Theorem B}
\newtheorem*{maincorollary}{Corollary C}

\theoremstyle{definition}
\newtheorem{definition}[theorem]{Definition}

\theoremstyle{remark}

\DeclareMathOperator{\crn}{cr}
\DeclareMathOperator{\jcr}{jcr}
\DeclareMathOperator{\bis}{ebis}
\newcommand{\Fpr}{\mathbb{P}}

\title{Asymptotics of the Meandering Number of a Cyclic Permutation}
\date{}

\author[Charles Daly]{Charles Daly}
\address{Max Planck Institute for Mathematics in the Sciences}
\email{\href{mailto:charles.daly@mis.mpg.de}{charles.daly@mis.mpg.de}}

\author[Diaaeldin Taha]{Diaaeldin Taha}
\address{Max Planck Institute for Mathematics in the Sciences}
\email{\href{mailto:Diaaeldin.Taha@mis.mpg.de}{diaaeldin.taha@mis.mpg.de}}

\dedicatory{In honor of Richard Schwartz on the occasion of his 60th birthday.}

\begin{document}

\begin{abstract}
A cyclic meander is an embedded oriented loop in the plane intersecting a fixed infinite line, or circle, transversely in a linearly ordered set of $2n$ points.  By keeping track of the order in which the loop visits these points, the cyclic meander induces a cyclic permutation on these marked points.  Correspondingly, given a permutation on $n$ letters, one can ask whether or not a cyclic meander induces the permutation in this manner, and if not, what is the most efficient way of doing so if we allow more points of intersection?  This process gives a way of associating to a permutation on $n$ letters a measurement of complexity of the permutation in question. The principal result of this work shows that the maximum of this quantity, the \emph{meander number}, over all cyclic permutations on $n$ letters, is bounded above and below quadratically in $n$.  This result resolves a conjecture of Schwartz~\cite{richtpss} in relation to his work on the topological salesman problem.  We conclude this work by proving the existence of families of cyclic permutations on $n$ letters whose meander numbers realize a continuum of growth rates between linear and quadratic.
\end{abstract}

\maketitle


\section{Introduction}
\label{sec:introduction}

\input{figures/meander-jordan-dictionary}

A \emph{cyclic meander} is a configuration of two simple closed embedded and oriented loops in the plane with non-empty transverse intersections.  One loop is frequently taken to be the $x$-axis in the plane, compactified at infinity, and oriented from left to right.  The other loop is depicted with all of its intersections occurring transversely along the $x$-axis.  This sort of configuration is well defined up to isotopies which fix the intersections and preserve transversality.  Frequently we will take them to be piecewise linear curves.  An example of a cyclic meander can be seen in Figure \ref{fig:mjd} (A) where the $x$-axis is drawn in bold and the other loop travels clockwise intersecting the axis in six points labeled $0, 1, \hdots, 5$ as ordered from the $x$-axis. Note this cyclic meander induces a permutation of the set $\{0,1,\hdots, 5\}$ given by following the dotted loop and keeping track of which intersection points are visited in order.  We thus say the permutation $(0,3,4,5,2,1)$ is \emph{meandric} as it comes from a cyclic meander.  It is clear not all cyclic permutations in $S_{N}$ are meandric, however, it is possible they may be realized along a cyclic meander if more points of intersection are allowed.  In this way, we say the cyclic meander induces the permutation.  
\\
\\
The study of meanders can be found across broad swathes of mathematical literature.  They appeared as early as in the work of Henri Poincar\'{e} where he used them to try to prove what is now known as the Poincar\'{e}-Birkhoff Theorem \cite{Poincare1912Sur}.  The combinatorial structure of meanders can be used to analyze compact foldings of closed polymer chains \cite{DIFRANCESCO199797}.  Much more recently Schwartz introduced the so-called \emph{topological salesman problem} wherein he considers the infimal length required to connect \emph{every} ordered collection of $N$ points on $[0,1]^{2}$ in a given order by an embedded path \cite{richtpss}.  In the work he relates this infimal length to the quantity $\mu_{N}$, which is minimal length of a cyclic permutation for which \emph{every} permutation on $N$ letters can be realized as a cyclic meander of length no more than $\mu_{N}$.  He shows that if the quantity is bounded below quadratically in the number of points $N$, then the infimal length of the topological salesman problem must grow on the order of $N^{3/2}$. He then shows that $\mu_N$ has a quadratic upper bound, and conjectures that it also has a quadratic lower bound --- a conjecture we refer to as the \emph{Schwartz quadratic meander number conjecture}. In this note, we show the conjecture holds for sufficiently large $N$.\\
\\
Before stating our main results, we separate two core objects in our work.  First, we record the two cyclic orders as a multigraph.  Then we introduce a Jordan-curve crossing number that measures how many unmarked intersections are needed to realize those two orders by two simple closed curves.  Our work relies on building a dictionary between those combinatorial and geometric objects.
\\
\\
We now fix notation. For each positive integer $n$ write the \emph{set}
\begin{equation*}
        [n]=\{0,1,\ldots,n-1\}.
\end{equation*}
A cyclic permutation $\sigma$ of $[n]$ is a permutation of $[n]$ with a single cycle.  Frequently we will identify $\sigma$ with its cycle notation below.    
\begin{equation*}
        \sigma=(\sigma_0,\sigma_1,\ldots,\sigma_{n-1})
\end{equation*}
This notation is understood up to cyclic rotation.  We note that every one of the $n!$ linear orderings $\sigma$ of $[n]$ induces a cyclic permutation given by the cycle above, and two linear orderings determine the same cyclic permutation if and only if they differ by a shift in the ordering.  One should think of a cyclic permutation as determining a Hamiltonian cycle graph on the set of vertices $[n]$ which tells you how to get from one vertex to the next.  We denote the standard linear ordering of $[n]$, $0 < 1 < \hdots < n-1$ by $\sigma_{0}$, and when we wish to emphasize the linear ordering on $[n]$ induced by a permutation $\sigma$, we shall decorate $[n]$ with the subscript $\sigma$, as in $[n]_{\sigma}$.  When we equip $[n]$ by the standard ordering induced by the identity permutation on $[n]$, we denote it by $[n]_{0}$.  
\\
\\
Let $N$ be an even number and consider $[N]_{0}$.  We define a \emph{matching} to be a collection of ordered pairs $\{(a_{i},b_{i})\}_{i=1}^{i=m}$ inside $[N]_{0}$ with $a_{i} < b_{i}$.  One should think of these as arcs all of which lie entirely above, or below, the real line connecting $a_{i}$ to $b_{i}$ for each $i$.  A matching is \emph{non-crossing} if it does not contain two pairs $\{a_{1},b_{1}\}$ and $\{a_{2},b_{2}\}$ with $a_{1}<a_{2}<b_{1}<b_{2}$.  \\
\\
A cyclic permutation $\sigma = (\sigma_{0},\sigma_{1},\hdots, \sigma_{N-1})$ is called \emph{cyclic meanderic} if the two alternating matchings
\begin{equation*}
        \bigl\{\{\sigma_0,\sigma_1\},\{\sigma_2,\sigma_3\},\ldots,
        \{\sigma_{N-2},\sigma_{N-1}\}\bigr\}
\end{equation*}
and
\begin{equation*}
        \bigl\{\{\sigma_1,\sigma_2\},\{\sigma_3,\sigma_4\},\ldots,
        \{\sigma_{N-1},\sigma_0\}\bigr\}
\end{equation*}
are non-crossing on $[N]_{0}$.  This definition is a purely combinatorial way to express the topological definition of a cyclic meander which is a pairing of two Hamiltonian cycles, one corresponding to the real axis going through $n$-labeled points in the order of $[n]$, and the other cycle going through the points in the order of $[n]_{\sigma}$.    
\\
\\
If $S\subset [N]_{0}$, the \emph{first-return permutation} of $\sigma$ on $S$ is obtained by reading $\sigma$ cyclically and deleting all letters outside $S$.  After relabelling the elements of $S$ increasingly by $[|S|]_{0}$, this gives a cyclic permutation of $[|S|]$. See the Figure \ref{fig:mjd} for an example where the cyclic meanderic permutation $(0,3,4,5,2,1)$ induces $(0,2,3,1)$ where $S = \{0,1,4,5\}$.  Following Schwartz, we recall the definition of the meander number of a cyclic permutation $\sigma$ from Section 3.2 \cite{richtpss}.

\begin{definition}\label{def:meander_number}
For a cyclic permutation $\sigma$ of $[n]$, define the \emph{meander number} $\mu(\sigma)$ to be the least even $N$ for which there exists a cyclic meandric permutation on $[N]$ whose first-return permutation on some $n$-element subset is $\sigma$.  For each $n$, we define $\mu_{n}$ to be the the largest $\mu(\sigma)$ for all $\sigma \in S_{n}$. 
\end{definition}
As we will be interested in asymptotics and growth rates, we briefly recall that a function $f(n)$ is $\Theta(g(n))$ if there are positive constants $c$ and $C$ for which $c g(n) \leq f(n) \leq C g(n)$ for sufficiently large $n$.  Part of our results shows that $\mu_{n} = \Theta(n^{2})$ as in Corollary \ref{cor:random-global} which shows Schwartz's quadratic meander number conjecture is true for sufficiently large $n$.  
\\
\\
We first establish an exact relationship between the meander number and the crossings of Jordan curves.  A \emph{PL Jordan realization} of $\sigma$ is a
pair of oriented piecewise-linear Jordan curves $A,B\subset S^2$, meeting
transversely in finitely many points, together with a distinguished $n$-element set
$S\subset A\cap B$ of marked intersections points in $A\cap B$ such that the marked points occur
in the natural cyclic order $[n]_{0}$ along $A$ and in the order $[n]_{\sigma}$, given by $\sigma_{0} < \sigma_{1} < \hdots < \sigma_{n}$, along $B$.  Let
\begin{equation*}
        \jcr(\sigma)=
        \min\{|A\cap B|-n:\ (A,B,S)\text{ is a PL Jordan realization of }\sigma\}.
\end{equation*}
We begin by arguing that no information is lost in passing from meanders to Jordan PL curves as in the proposition below.
\begin{mainformula}[PL Jordan realization formula]
For every cyclic permutation $\sigma$ of $[n]$,
\begin{equation*}
        \mu(\sigma)=n+\jcr(\sigma).
\end{equation*}
\end{mainformula}


Second, to a cyclic permutation $\sigma$ we associate the following multigraph.
\begin{definition}[Cycle-pair multigraph]\label{def:cpmg}
For a cyclic permutation $\sigma$ of $[n]$ with $n\ge 2$, the \emph{cycle-pair
multigraph} $G_{\sigma}$ is the $4$-valent multigraph on vertex set $[n]$ formed as the union of
two Hamiltonian cycles: the \emph{natural cycle}
\begin{equation*}
        C_0=(0,1,\ldots,n-1,0),
\end{equation*}
and the \emph{$\sigma$-cycle}
\begin{equation*}
        C_\sigma=(\sigma_0,\sigma_1,\ldots,\sigma_{n-1},\sigma_0).
\end{equation*}
\end{definition}

This multigraph encodes the data of the two Jordan curves we wish to draw on the plane with the least amount of additional intersections.  Studying $G_{\sigma}$ will allow us to use graph-theoretic tools to relate the Jordan crossing number of $\sigma$ to a quantity that measures in some sense how much the $G_{\sigma}$ is expanding as in Equation \ref{eq:bis-wid}.  
\\
\\
Every PL Jordan realization of $\sigma$ gives a drawing of $G_{\sigma}$, and therefore the Jordan crossing number $\jcr(\sigma)$ is bound from below by the crossing number of $G_{\sigma}$.  Recall the crossing number of a multigraph $G$, denoted by $\crn(G)$, is the minimum number of edge crossings required to realize the multigraph in the plane.  The main content of this work in Section \ref{sec:lower-bound} provides a lower bound on this crossing number for $G_{\sigma}$ for certain permutations $\sigma$ as in Theorem \ref{thm:random-path}.  Moreover we show that these permutations are generic in the sense that most permutations satisfy this lower bound.  
\\
\\
The upper-bound for $\jcr(\sigma)$ is constructive.  We break the cyclic permutation into \emph{tiles}, as in Lemma \ref{lem:tile}, and bound the number of crossing above on each tile.  The idea to induce the cyclic permutation is as follows.  We form a grid of $n(n+1)$-points and connect the points along lines of fixed $y$-coordinates by going from left to right then right to left and alternating in this manner.  This will form our curve $A$.  We do the same process but now instead connect points along fixed $x$-coordinates going from bottom to top then top to bottom and alternating in this manner.  Figure \ref{fig:permutation-tile} illustrates this process.  This method will produce a quadratic upper bound in $n$ on the Jordan crossing number of $\sigma$.  
\\
\\
Putting these pieces together gives the main quantitative result of the paper.
\begin{mainintermediate}[All intermediate scales]
Let $L(n)$ be an integer-valued function such that
\begin{equation*}
        L(n)\to\infty,
        \qquad
        1 \le L(n) \le n.
\end{equation*}
Then there are cyclic permutations $\sigma_n$ of $[n]$ such that
\begin{equation*}
        \mu(\sigma_n)=n+\Theta(nL(n)).
\end{equation*}
\end{mainintermediate}

By choosing the block size $L(n)$, one can realize all scales between the linear and quadratic regimes covered by this theorem. Taking $L(n)=n$ gives the quadratic-growth consequence relevant to Schwartz's conjecture.

\begin{maincorollary}
There are constants $c,C>0$ such that for all $n$,
\begin{equation*}
        c n^2
        \le
        \mu_{n}
        \le
        C n^2.
\end{equation*}
\end{maincorollary}


\subsection*{AI Methodology} 

The authors used AI systems during the research and writing process.  Experiments
with AlphaEvolve \cite{DBLP:journals/corr/abs-2506-13131,georgiev2025mathematical} were used to search for and analyze meandric permutations. These
experiments helped the authors develop a more nuanced understanding of meandric permutations and informed some of the constructions in this work.  The
results of the paper, however, are proved independently of these experiments.
The authors also used large language models, including ChatGPT 5.5, Claude Opus 4.7, and Gemini 3.1, as research assistants, mainly to explore potentially relevant existing results for this work. Only a small portion of those explored results appear in this manuscript. The authors directed all use of these tools, then wrote and verified the proofs, and produced the figures with some assistance from the models.

\subsection*{Organization}
This note is organized as follows.  Section \ref{sec:pl-jordan} defines PL Jordan realizations and proves the exact formula $\mu=n+\jcr$.  Section \ref{sec:upper-bound} gives the rectangular wiring construction and the resulting upper bounds.  Section \ref{sec:lower-bound} proves the block lower bound and constructs crossing-rich tiles. Section \ref{sec:growthrates} assembles these pieces to obtain the intermediate-regime theorem and the random permutation corollaries.

\subsection*{Acknowledgements}

 We thank the AlphaEvolve team at Google DeepMind for access to the system and for discussions about auditable AI-assisted mathematical discovery.


\section{PL Jordan realizations}
\label{sec:pl-jordan}

This section recasts the meander number problem in a convenient geometric-topological model: studying the intersection pattern of two piecewise-linear Jordan curves on the sphere. Lemma~\ref{lem:exact} makes the connection between meander numbers and piecewise-linear Jordan curves exact, and Lemma~\ref{lem:cr-lower} relates the relevant number of intersections of two piecewise-linear Jordan curves to crossing numbers of multigraphs.  This allows us to use graph-theoretic machinery, which we do in Section~\ref{sec:lower-bound}.

\begin{definition}
Let $\sigma$ be a cyclic permutation of $[n]$.  A \emph{PL Jordan
realization} of $\sigma$ is a pair of oriented piecewise-linear Jordan
curves $A,B\subset \mathbb{S}^2$ meeting transversely in finitely many points, together with a distinguished
subset $S\subset A\cap B$, $|S|=n$,
such that, after labelling the points of $S$ by $0,1,\ldots,n-1$ in
their cyclic order along $A$, their cyclic order along $B$ is $\sigma$. The \emph{PL Jordan crossing number} is
\begin{equation}
        \jcr(\sigma)
        =
        \min\{|A\cap B|-n:\ (A,B,S)\text{ is a PL Jordan realization of }
        \sigma\},
\end{equation}
\end{definition}

The next proposition provides a basic dictionary for passing between meanders and Jordan curves.
\begin{lemma}\label{lem:exact}
For every cyclic permutation $\sigma$ of $[n]$,
\begin{equation}
        \mu(\sigma)=n+\jcr(\sigma).
\end{equation}
\end{lemma}

\begin{proof}
We prove that the possible excesses $N-n$ coming from a cyclic meander are exactly the possible excesses $|A\cap B|-n$ coming from
PL Jordan realizations.
\\
\\
First suppose that $\sigma$ is a cyclic meandric permutation on $[N]$ inducing
$\sigma$ on a subset $S\subset [N]$.  Place the points $0,1,\ldots,N-1$ in order on an oriented circle $A\subset \mathbb{S}^2$.  Draw the first alternating matching by pairwise disjoint arcs in one complementary disc of $A$, and draw the second alternating matching by pairwise disjoint arcs in the other complementary disc.  The union of these arcs is an embedded connected  $2$-valent  graph: abstractly, it is the cycle $(\sigma_0,\sigma_1,\ldots,\sigma_{N-1},\sigma_0)$. Near each point of $A$, one incident arc lies on one side of $A$ and the other incident arc lies on the other side.  Thus, after an arbitrarily small local smoothing at the $N$ points of $A$, the union becomes a piecewise-linear Jordan curve $B$ crossing $A$ transversely at precisely those $N$ points.  The distinguished intersections are the points of $S$.  Their order along $A$, after increasing relabelling of $S$, is $0,1,\ldots,n-1$, and their first-return order along $B$ is $\sigma$. Thus every meandric permutation with $N$ letters gives a PL Jordan realization with $N-n$ unmarked intersections.
\\
\\
Conversely, let $(A,B,S)$ be a PL Jordan realization and let $N=|A\cap B|$. The curve $A$ separates $\mathbb{S}^2$ into two discs.  Since $B$ crosses $A$ transversely whenever it meets $A$, the successive arcs of $B$ alternate between the two complementary discs.  Thus $N$ is even as the curves are closed.  Arcs lying in the same disc are pairwise disjoint because $B$ is embedded.  Therefore each of the two alternating matchings determined by $\sigma$ are non-crossing as defined in Section \ref{sec:introduction}. Label all points of $A\cap B$ by $0,1,\ldots,N-1$ in their cyclic order along $A$, and let $\sigma=(\sigma_0,\ldots,\sigma_{N-1})$ be their cyclic order along $B$. The permutation $\sigma$ is cyclic meandric, and its first-return permutation on the labels corresponding to $S$ is $\sigma$.
\\
\\
The two constructions preserve the excess $N-n$ exactly and taking minima gives the
claim.
\end{proof}

In light of Lemma \ref{lem:exact}, we may think about this problem as drawing two Jordan curves intersecting transversally in the plane at a set of prescribed points in a prescribed order and by adding as few additional intersections as possible.  This naturally leads us to think of the associated cycle-pair multigraph $G_{\sigma}$ as in Definition \ref{def:cpmg} coming from the permutation $\sigma$ in question.   \\
\\
In what follows, we assume that all drawings of a multigraph have vertices that are distinct points, edges are simple arcs, no edge passes through a non-incident vertex, crossings are transverse interior intersections of two distinct edges, and, there are no triple crossings.  Parallel edges are allowed and are drawn as distinct arcs.  We assume these hypotheses are always satisfied for any drawing of a multigraph.   \\
\\
Let $G$ be a multigraph.  The \emph{crossing number} $\crn(G)$ is the minimum number of crossings among all such drawings of $G$.  Crossings at common vertices $G$ are not counted. That is to say, edges may meet at their shared endpoints, and all other intersections are counted as crossings.  With these assumptions established, we are now able to prove our next proposition which relates the Jordan crossing number of a permutation to the crossing number of a permutations associated cycle-pair graph $G_{\sigma}$.  

\begin{lemma}\label{lem:cr-lower}
For every cyclic permutation $\sigma$ of $[n]$, with $n\ge 2$,
\begin{equation*}
        \crn(G_{\sigma}) = \jcr(\sigma)=\mu(\sigma)-n.
\end{equation*}
\end{lemma}

\begin{proof}
Take a PL Jordan realization $(A,B,S)$ where $n = |S|$.  Note the realization of $(A,B,S)$ has $N$-vertices, $n$ of which are marked and $N-n$ of which are unmarked.  If we simply remove the unmarked vertices from this realization, while retaining the edges, this will give a drawing of $G_{\sigma}$ where the unmarked vertices are precisely the crossings so $\crn(G_{\sigma}) \leq \jcr(\sigma)$.  \\
\\
Conversely, if we have a drawing of $G_{\sigma}$, we may orient the two Hamiltonian cycles $C_{0}$ and $C_{\sigma}$ in accordance with their orderings on $[n]$ and $[n]_{\sigma}$ respectively.  Because there are only finitely many intersections between these cycles, we may assume transversality of the intersections of the drawing of $G_{\sigma}$ at both the vertices of $G_{\sigma}$, and, their crossings.  This will define a collection of finitely many disjoint arcs in the plane meeting at finitely many points corresponding to these intersections.  One may isotope the edges to be unions of line segments in such a manner that none of the edges intersect at new points throughout the isotopy.  Take as $A$ and $B$ the union of edges corresponding to $C_{0}$ and $C_{\sigma}$ respectively.  The vertex set of $G_{\sigma}$ may be taken as $S$, and thus this yields a PL Jordan realization of $(A,B,S)$ where $|A\cap B| - n = \crn(G_{\sigma})$.  Thus $\jcr(\sigma) \leq \crn(G_{\sigma})$ concluding our proof.  
\end{proof}


\section{The upper bound}
\label{sec:upper-bound}

This section describes constructions that provide controlled upper bounds for the PL Jordan crossing number $\jcr$.  We first build what we refer to as permutation tiles as in Lemma~\ref{lem:tile}.  These are rectangular tiles in which two PL arcs realize an arbitrary ordering of a block of integers with a controlled number of intersections.  We then connect these tiles as in Proposition \ref{prop:block-upper}, producing two Jordan curves without creating new intersections. The resulting upper bound on the Jordan crossing number grows quadratically.  We remark that while this upper bound is weaker than the original bound Schwartz provided in Lemma 3.3 \cite{richtpss}, the construction used here are distinct and could find applications in other areas such as tennis racket stringing. 
\begin{lemma}[Permutation tiles]\label{lem:tile}
Let $a$ be a non-negative integer, and $L$ a positive integer.  Define $I=\{a,a+1,\ldots,a+L-1\}\subset\mathbb N$, and let $\sigma=(\sigma_{0},\sigma_{1},\ldots,\sigma_{L-1})$ be an ordering of $I$.  There is a rectangle $Q$ and two properly
embedded piecewise-linear arcs $A_\sigma,B_\sigma\subset Q$
with endpoints on $\partial Q$, meeting transversely, such that: $(i)$ the initial and terminal endpoints $A^-,A^+$ of $A_\sigma$ lie on the top side of
$Q$, with the initial endpoint $A^-$ to the left of the terminal endpoint $A^+$;
$(ii)$ the initial and terminal endpoints $B^-,B^+$ of $B_\sigma$ lie on the bottom side of
$Q$, with the initial endpoint to the left of the terminal endpoint;
$(iii)$ $A_\sigma\cap B_\sigma$ has exactly $L(L+1)$ points;
$(iv)$ $L$ of these points are marked and labelled by the elements of $I$;
$(v)$ the marked points occur along $A_\sigma$ in the order
$(a,a+1,\ldots,a+L-1)$; and
$(vi)$ the marked points occur along $B_\sigma$ in the order $\sigma$. We
call such a tuple $(Q,A_\sigma,B_\sigma)$, together with its marked labelled
intersections, a \emph{permutation tile} for $\sigma$.
\end{lemma}

\input{figures/permutation-tile}

\begin{proof}
This proof is best understood with an accompanying picture.  Figure \ref{fig:permutation-tile} illustrates the construction.  Pick a sufficiently large rectangle $Q$ in $\mathbb{R}^{2}$ of side lengths say $(2L+2)$ which contains big enough parts of the line segments $y = 0, y = 1, \hdots, y = L$ and $x = a, x = a + 1, \hdots, x = a+L$.  The bottom left corner of $Q$ should be placed in such a way that the point $(a,0)$ is in the interior of $Q$ but close to the bottom left corner of $Q$.  \\
\\
Starting along the $x$-axis, $y=0$, label the points $(a,0), (a+1,0), (a+2,0),\hdots, (a+L-1,0), (a+L,0)$ by $0,1,2,\hdots, L-1,L$ in this order, i.e. from left to right.  Go up one unit to the line $y = 1$ and label the points $(a+L,1), (a+L-1,1), \hdots, (a+2,1), (a+1,1), (a,1)$ by $L+1, L+2, L+3, \hdots, 2L, 2L+1$ in this order, i.e. from right to left.  Go up one unit to the line $y = 2$ and label the points $(a,2), (a+1,2), (a+2,2),\hdots, (a+L-1,2), (a+L,2)$ by $2L+2, 2L+3, 2L+4,\hdots, 3L+1, 3L+2$ in this order, i.e. from left to right.  Continue this process in an alternating manner going through all the lines $y = 0, y = 1, \hdots , y = L-1$.  This forms a grid of $L(L+1)$ points inside the interior of $Q$.  We now construct a path of line segments going through each point in specific orders.  \\
\\
To construct the arc $A_{\sigma}$, start at a point $A^{-}$ on the top edge of $Q$ that is to the left of $x = a$.  Draw a line segment vertically down from this point to the $x$-axis $y = 0$.  This point of intersection will lie to the left of $(a,0)$.  Draw the line segment from this point of intersection to a point to the right of $(a+L,0)$ in the interior of $Q$.  Now draw the vertical line segment from this point to the line $y = 1$.  The point of intersection will lie to the right of $(a+L,1)$.  Draw the line segment from this point of intersection to a point to the left of $(a,1)$ in the interior of $Q$.  Continue this process snaking through interior of $Q$ until you run through all the $L(L+1)$ points in order.  This last segment will produce a point on the line $y = L-1$ either to the right or left of the last point corresponding to $L(L+1)$ depending on the parity of $L$.  Complete the arc by going to a point on the top edge of $Q$ that is to the right of $x = a+L$ via vertical and horizontal line segments.  This defines $A_{\sigma}$ which enters from the top left of $Q$, snakes through the $L(L+1)$ points $\{0,1,2,\hdots, L(L+1)-1\}$ in this order, horizontally, and exits at a point $A^{+}$ close to the top right of $Q$.  \\
\\
To construct the arc $B_{\sigma}$, we start at a point $B^{-}$ on the bottom edge of $Q$ that is to the left of $x = a$.  Draw a small upward vertical line segment from $B^{-}$, then a horizontal line segment to the right until you hit the line $x = a$.  Draw the line segment from this point of intersection to a point above $(a,L-1)$ in the interior of $Q$.  Now draw a horizontal line segment from this point to the line $x = a+1$.  The point of intersection will lie above $(a+1,L-1)$.  Draw the line segment from this point of intersection to a point below $(a+1,0)$ in the interior of $Q$.  Continue in this process snaking through the interior of $Q$ until you run through all the $L(L-1)$ points.  In this order they will appear as $(a,0), (a,1), \hdots, (a,L-1)$, then $(a+1,L-1), (a+1,L-2), \hdots, (a+1,0)$, so on and so forth.  The last segment will produce a point on the line $x = L$ either above $(a+L,L-1)$ or below $(a+L,0)$ depending on the parity of $L$.  Complete the arc by going to a point on the bottom edge of $Q$ that is to the right of $x = a+L$ via vertical and horizontal line segments disjoint from $A_{\sigma}$.  This defines $B_{\sigma}$ which enters from the bottom of $Q$ and snakes through the $L(L+1)$ points, vertically, and exits at a point $B^{-}$ close to the bottom right of $Q$.  Figure \ref{fig:permutation-tile} provides a picture of $A_{\sigma}$ and $B_{\sigma}$ for $L = 4$.  \\
\\
By construction, properties $(i) - (iii)$ of the lemma are satisfied.  We now pick from the $L(L+1)$ points of $A_{\sigma} \cap B_{\sigma}$ a subset to represent the interval $I = \{a,a+1,\hdots, a+L-1\}$.  To this end, on the line $x = a$, reading from bottom to top, pick the $\sigma_{0}$-th vertex from the $L$-vertices on the line $x=a$.  For $x=a+1$, again reading from bottom to top, pick the $\sigma_{1}$-th vertex from the $L$ vertices on the line $x = a+1$.  Continue in this fashion for each $i$ between $0$ and $L$.  This selects the subset $S \subset [L(L+1)]$ of $L$-marked points from $A_{\sigma} \cap B_{\sigma}$.  \\
\\
We label the marked points $S$ by elements of $I$ by the following convention.  Follow the path $A_{\sigma}$ from $A^{-}$ to $A^{+}$.  The first element of $S$ that is seen is labeled by $a$.  The second element seen is labeled by $a+1$.  Continue in this fashion.  Notice this labeling is precisely by height reading from bottom to top.  That is to say, $a$ is the unique element of $S$ lying on the line $y = 0$, and $a+1$ is the unique element of $S$ lying on the line $y = 1$, so on and so forth.  By construction, the set of marked points $S$ along $A_{\sigma}$ occur in the order $(a,a+1,\hdots, a+L-1)$.  \\
\\
On the other hand, consider what happens when we traverse $B_{\sigma}$.  The first line segment will be directed upwards at $x = a$.  By construction, the corresponding point of $S$ is the $\sigma_{0}$-th vertex which is $\sigma_{0} \in S$.  The second line segment will be directed downwards at $x = a+1$.  By construction, the corresponding point of $S$ is the $\sigma_{1}$-th vertex which is $\sigma_{1} \in S$.  Continue in this fashion.  This means the marked points $S$ occur along $B_{\sigma}$ in the order of $\sigma$.  Figure \ref{fig:permutation-tile} provides a picture of this construction.  Note reading left to right, the set of marked points are chosen at the height $y = 1, y= 2$ and $y = 0$.  Traveling along $A_{\sigma}$, we index the marked green points by $\{0,1,2\}$, and along $B_{\sigma}$ these points occur in the order $(1,2,0)$.  This concludes the proof.
\end{proof}

Next we explain how connect several permutation tiles together to obtain a pair of Jordan curves that intersect in a controlled behavior.  To do so, we first introduce some notation.  Let $[n] = I_{1} \sqcup \hdots \sqcup I_{q}$ be a partition of $[n] = \{0,1,\hdots, n-1\}$ into $q$ disjoint consecutive intervals of $[n]_{0}$ where $I_{1} < I_{2} < \hdots < I_{q}$ and $L_{i} := |I_{i}| \geq 1$ for each $1 \leq i \leq q$.  Choose a linear ordering on each of the $I_{i}$'s and concatenate them together according to the order $I_{1} <  I_{2} < \hdots < I_{q}$ to get a linear ordering on all of $[n]$.  We denote the resulting linear order on $[n]$ by $\sigma$.  \\
\\
Each linear ordering $\sigma_{i}$ on $I_{i}$ determines a cyclic permutation by taking each element to its immediate successor under the order of $I_{i}$.  That is, $\sigma_{i}$ is the permutation induced by $(\sigma_{i}(a_{i}), \sigma_{i}(a_{i}+1) , \hdots ,\sigma_{i}(a_{i}+L_{i}-1))$ where $I_{i} = \{a_{i},a_{i}+1,\hdots, a_{i}+L_{i}-1\}$.  We take care here to remember the base-point $a_{i}$ of each $I_{i}$.  By concatenating all of the linear orderings $\sigma_{i}$ together, we then determine a cyclic permutation on $[n]$ which we abusively denote by $\sigma$.  The permutation has the property that it runs through the intervals $I_{i}$ in the cyclic order determined by $i = 1, 2, \hdots, q$.  We call a cyclic permutation that arises from this construction a \emph{cyclic interval permutation}.  We now provide a topological realization of this process in terms of the tiles.

\begin{proposition}[Block wiring upper bound]\label{prop:block-upper}
Let
\begin{equation*}
        [n]=I_1\sqcup I_2\sqcup\cdots\sqcup I_{q}
\end{equation*}
be a partition into consecutive intervals, with $L_i=|I_i| \geq 1$.  Let $\sigma_{i}$ be a linear ordering for each $I_{i}$, and let $\sigma$ denote the cyclic permutation obtained from the linear ordering on $[n]$ by concatenating each of the orderings $I_{1} < I_{2} < \hdots < I_{q}$.  We then have the following upper bound on the Jordan crossing number of $\sigma$.
\begin{equation*}\label{eq:jcr_squarebnd}
        \jcr(\sigma)\le \sum_{i=1}^{q} L_i^2.
\end{equation*}
\end{proposition}

\begin{proof}
For each interval $I_{i}$, apply Lemma~\ref{lem:tile} with $(I_{i},\sigma_{i})$ to obtain tiles $Q_{i}$ and pairs of arcs $(A_{\sigma_{i}}, B_{\sigma_{i}})$ for $1 \leq i \leq q$ which satisfy the conclusions of Lemma~\ref{lem:tile}.  In particular, each tile $Q_{i}$ contains $L_{i}(L_{i}+1)$ points, $L_{i}$ of which are marked and $L_{i}^{2}$ of which are unmarked.  \\
\\
Place these tiles in the plane from left to right in the order determined by $I_{1} < \hdots < I_{q}$ with no overlap and translated if necessary.  Connect the tiles together as according to Figure \ref{fig:corridor}.  Denote the resulting cycles by the connecting $A_{\sigma_{i}}$'s and $B_{\sigma_{i}}$'s by $A$ and $B$ respectively.  This supplies us with a pair of PL Jordan curves intersecting at $\sum_{i=1}^{q} L_{i}(L_{i}+1)$ with precisely $\sum_{i=1}^{q} L_{i} = n$ marked points.  By construction, the pair $(A,B)$ is a PL realization of $\sigma$ whose distinguished marked subset is the union of the marked subsets of each of the tiles from the $Q_{i}$.  Along $A$ these points occur in the natural order, and along $B$, they occur in the order of $\sigma$.  Because the number of marked points is precisely $n$, the number of unmarked points is equal to $\sum_{i=1}^{q} L_{i}^{2}$.  Consequently the Jordan crossing number of $\sigma$ is bounded above as in Equation \ref{eq:jcr_squarebnd}.  
\end{proof}

\input{figures/corridor}

Proposition \ref{prop:block-upper} provides us with an upper bound the on the complexity of the Jordan crossing number of a cyclic permutation $\sigma$ which arises as a concatenation of smaller cyclic permutations.  The proposition shows that if one creates such a permutation in this manner, the Jordan crossing number grows at most as a sum of quadratics in the length of each subinterval, and one can create permutations of $[n]$ where the complexity of the full permutation is spread out across the smaller blocks.  The case of main interest to us is in the case of a single block and we conclude this section by addressing a case of this explicitly.  We note however as mentioned before Schwartz provides a more refined construction that gives the tighter bound of $\mu_{n}$, as in Definition \ref{def:meander_number}, and shows it grows at most approximately as $n^{2}/2$.  

\begin{corollary}\label{cor:global-upper}
For every cyclic permutation $\sigma$ of $[n]$,
\begin{equation*}
        \jcr(\sigma)\le n^2.
\end{equation*}
Consequently
\begin{equation*}
        \mu_{n} \le n^2+n.
\end{equation*}
\end{corollary}

\begin{proof}
Apply Proposition~\ref{prop:block-upper} with the single block $I_1=[n]$,
then use Lemma~\ref{lem:exact}.
\end{proof}

\section{The lower bound}
\label{sec:lower-bound}

In the previous section we established an upper bound on the Jordan crossing number $\jcr(\sigma)$ for a cyclic permutation $\sigma$.  It is the purpose of this section to provide a lower bound.  The methods used are probabilistic and should be thought of addressing the case for $\sigma$ permuting a large number of elements.  We provide an outline of the contents of this section before proceeding.  \\
\\
Recall for a cyclic permutation $\sigma$, we have the cycle-pair multigraph $G_{\sigma}$ as in Definition \ref{def:cpmg}.  Lemma \ref{lem:cr-lower} says the Jordan crossing number of the permutation $\sigma$ is equal to the crossing number of the multigraph $G_{\sigma}$, and therefore it suffices to bound the latter quantity instead.  We first reduce the crossing number of $G_{\sigma}$ to that of a smaller sub-multigraph as in Definition \ref{def:ppmg} allowing us to work with Hamiltonian \emph{paths} instead of \emph{cycles} as in Proposition \ref{prop:block}.  The difference between this is more or less the difference between linear orderings of $[n]$ and cyclic orderings of $[n]$, the former of which is preferable as the counting arguments are more easily stated in the linear case.
\\
\\
We then bound this crossing number from below in terms of its \emph{edge bisection width} as in Lemma \ref{lem:bisection-crossing}.  Informally, the edge bisection width is a measurement of how much the multigraph is `expanding' in some sense.  We then show through a graph theoretic argument, whose techniques are owed to Lipton-Tarjan, that if a graph has a large edge bisection width, then it must have a large crossing number.  This is the content of Lemma~\ref{lem:bisection-crossing}.  \\
\\
To bound the crossing number below quadratically, it then suffices to show that there are permutations that have large edge bisection width in their associated multigraph.  Our methods are non-constructive and rely on counting estimates which are addressed throughout several lemmas in this section, and ultimately resulting in Theorem \ref{thm:random-path}.  In some sense, these counting arguments can be boiled down to the following observation.  If you take an interval in $[n]_{0}$ of length about $n/3$, and apply a random permutation to it, the odds that it is still an interval in $[n]_{0}$ are very small.  To begin, we introduce a definition that is the path-analogue of Definition \ref{def:cpmg} for a permutation $\sigma$.  

\input{figures/scrambled}

\begin{definition}[Path-pair graph]\label{def:ppmg}
For a cyclic permutation $\sigma$ of $[n]$ with $n \geq 2$, the \emph{path-pair multigraph} $H_{\sigma}$ is the multigraph on the vertex set $[n]$ formed as the union of the two Hamiltonian paths: the natural path 
\begin{equation*}
        P_0=(0,1,\ldots,n-1),
\end{equation*}
and the \emph{$\sigma$-path}
\begin{equation*}
        P_\sigma=(\sigma_0,\sigma_1,\ldots,\sigma_{n-1}).
\end{equation*}
Note that $H_{\sigma}$ is obtained from the cycle-pair multigraph $G_{\sigma}$, but with two edges removed.  Thus $H_{\sigma}$ is $4$-valent except at most $4$-points.  Some examples are provided in Figure \ref{fig:scrambled}.  
\end{definition}

For a partition of $[n]$ into consecutive intervals, $[n] = I_{1} \sqcup \hdots \sqcup I_{q}$, and linear orderings on each $I_{i}$ with associated permutations $\sigma_{i}$, we can define a collection of sub-path-pair multigraphs $H_{\sigma_{i}}$ of the cycle-pair multigraph $G_{\sigma}$ where $\sigma$ is the permutation of $[n]$ induced by concatenating the linear orderings together as was done in the paragraph before Proposition \ref{prop:block-upper}.  In this context, we are able to bound the crossing number of $G_{\sigma}$ below by the sum of the crossing numbers of the various $H_{\sigma_{i}}$ as seen in the following proposition.  

\begin{proposition}[Block Crossing Bounds]\label{prop:block}
Let $[n] = I_{1} \sqcup \hdots \sqcup I_{q}$ be a partition into consecutive intervals of length $L_{i} \geq 2$ and $\sigma_{i}$ linear orderings on each $I_{i}$.  Form the various path-pair multigraphs $H_{\sigma_{i}}$ and the linear ordering $\sigma$ on $[n]$ by concatenating each of the linear orderings $I_{i}$ together.  Denote the resulting permutation of $[n]$ by $\sigma$ as well.  We then have the following inequalities below.
\begin{equation*}
\sum_{i=1}^{q} \crn(H_{\sigma_{i}}) \leq \crn(G_{\sigma}) \leq \sum_{i=1}^{q} L_{i}^{2}
\end{equation*}
\end{proposition}

\begin{proof}
The upper bound is Proposition~\ref{prop:block-upper}.  For the lower bound, observe that $G_{\sigma}$ contains the multigraphs $H_{\sigma_{j}}$ as pairwise edge-disjoint submultigraphs.  In any drawing of $G_{\sigma}$, the restriction to the edges of $H_{\sigma_{j}}$ is a drawing of $H_{\sigma_{j}}$.  Therefore the number of crossings whose two participating edges both lie in $H_{\sigma_{j}}$ is at least $\crn(H_{\sigma_{j}})$.  These internal crossing pairs are disjoint for different $j$, because the $H_{\sigma_{j}}$'s are edge-disjoint as edge multisets.  Hence
\begin{equation*}
        \sum_{j=1}^q \crn(H_{\sigma_{j}}) \leq \crn G_{\sigma}  
\end{equation*}
\end{proof}

For a finite multigraph $H$ and $U\subseteq V(H)$, let $\partial_H U$ be the multiset of edges with one endpoint in $U$ and one endpoint outside $U$.  The size of $\partial_H U$ gives us a measurement of how `connected' $U$ is in $H$ with smaller values being more `connected'.  We say a set of vertices $U \subset V(H)$ is called \emph{balanced} if $|V(H)|/3 \leq |U| \leq 2|V(H)|/3$, that is to say, $U$ `sees' about as many vertices of $H$ as it misses.  Finally, following F{\"u}redi and K{\"u}ndgen, we define the \emph{edge bisection width} below which intuitively can be thought of as a measure of how much the graph $H$ is expanding \cite{FurediKundgen2006}.  
\begin{equation}\label{eq:bis-wid}
        \bis(H)
        =
        \min\{|\partial_H U|:\ U\subseteq V(H),\ |V(H)|/3\le |U|\le 2|V(H)|/3\}.
\end{equation}

We use the weighted Lipton--Tarjan planar separator theorem in the following standard form \cite[Theorem~4]{LiptonTarjan}: Every planar graph with $m$-vertices and nonnegative vertex weights of total weight $1$ has a vertex set $Z$ with $|Z|\le 2\sqrt{2}\sqrt m$ such that every component of the complement has weight at most $2/3$.  The next lemma is a standard separator-to-crossing-number argument, going back to Leighton~\cite[Theorem~7-1]{leighton1983complexity}; see also Pach--Shahrokhi- Szegedy~\cite[Theorem~2.1]{pach1994applications} and Djidjev~\cite{djidjev2001improved,djidjev2003crossing}.  The intuition behind it is that a drawing of a multigraph with few crossings planarizes to a planar graph with a small separator and such a separator would cut too few edges if the original graph had large edge bisection width.  Consequently the crossing number of the multigraph must be sufficiently large. 

\begin{lemma}[Bisection forces crossings]\label{lem:bisection-crossing}
Let $H$ be a multigraph on $m$ vertices, with maximum degree at most
$\Delta\ge 2$.  The crossing number of $H$ is bounded below quadratically in the edge bisection width as in following inequality.   
\begin{equation*}
        \crn(H)
        \ge
        \frac{\bis(H)^2}{8\Delta^2}-m.
\end{equation*}
In particular, if for some fixed $\eta > 0$, we have both $\bis(H)\ge \eta m$ and $m \geq 16\Delta^{2}/\eta^{2}$, then
\begin{equation*}
        \crn(H)\ge \frac{16\Delta^{2}}{\eta^{2}}m^2
\end{equation*}
\end{lemma}

\begin{proof}
Let $x=\crn(H)$, and choose a drawing of $H$ with $x$ crossings in general position.  Planarize the drawing by replacing every crossing by a new vertex.  The resulting planar multigraph $P$ has $m+x$ vertices.\\
\\
We next construct a vertex subset $Z$ so that when we remove a certain subset of edges of $H$ defined in terms of $Z$ via Lipton--Tarjan, we obtain a new multigraph that has many components, but, the same number of vertices as $H$ and relate the number of vertices to the edge bisection width.  To this end, we define the following `cost' function on the vertex set $V(P)$.  Assign the weight $1/m$ to each original vertex and weight $0$ to each crossing vertex.  The total cost is $1$.  \\
\\
Form a new graph $P'$ by taking $P$ and removing all loops, and, for every pair of vertices with multiple edges connecting them, we remove all but one.  Passing from $P$ to $P'$ yields a simple planar graph and does not change the vertex set, nor the vertex weights.  Most importantly, for any subset $S$ of the vertex set, the components of $P-S$ and $P'-S$ are equal by construction.  We may thus apply the weighted planar separator theorem of Lipton--Tarjan to $P'$ and obtain a subset $Z$ of vertices from $P'$ with
\begin{equation*}
        |Z|\le 2\sqrt{2} \sqrt{m+x}
\end{equation*}
such that every component of $P'-Z$ has `cost' at most $2/3$.  By definition of our cost function, this means each component of $P'-Z$ has at most $2m/3$ of the original vertex set of $H$.  Because the components of $P'-Z$ and $P-Z$ are the same, this means that each component of $P-Z$ has at most $2m/3$ of the original vertex set $H$.  \\
\\
Our vertex set $Z$ consists of two types of vertices which we will refer to as `original vertices' and the vertices added corresponding to crossings which we will refer to as `crossing vertices.'  Take $H$ and form a new multigraph $H'$ defined by removing from $H$ each edge that is incident to an original vertex of $Z$, and, removing each edge from $H$ that passes through a crossing vertex of $Z$.  An original vertex of $Z$ accounts for at most $\Delta$-deleted edges of $H$ by hypothesis on the degree of the multigraph whereas every crossing vertex of $Z$ accounts for at most two deleted edges of $H$.  Since $\Delta\ge 2$, the total number of deleted edges is at most
\begin{equation*}
        \Delta |Z|
        \le
        2\sqrt{2} \Delta\sqrt{m+x}.
\end{equation*}
We emphasize that $H'$ comes from $H$ by deleting edges, not vertices, thus the vertex set of $H'$ is the same as $H$.  Note that each original vertex $v \in Z$ is an isolated component of $H'$ because each edge that was adjacent to it in $H$ was deleted.  Pick instead a component corresponding to $v \in V(H)-Z$.  We claim this component is contained in a component of $P-Z$.  To this end, if $v, w \in V(H)-Z$ are connected in $H'$, then we may connect them by edges of $H$ such that each edge is not incident to an original vertex of $Z$, and, each edge does not pass through a crossing vertex of $Z$.  By construction each such edge is also an edge of $P-Z$, and thus $v, w$ are connected in $P-Z$.  Hence every component of $H'$, except possibly isolated original vertices from $Z$, lies inside a component of $P-Z$.  Therefore, every component of $H'$ contains at most $2m/3$ vertices, or, is an isolated single vertex corresponding to an original vertex of $Z$.  
\\
\\
Choose a union $U$ of components of $H'$ so that $m/3\le |U|\le 2m/3$, i.e. $U$ is a balanced subset.  This can always be done because if some component has size at least $m/3$, simply choose that.  Otherwise, one may add components until the vertex count first reaches at least $m/3$ at which point the total number of vertices is still at most $2m/3$.  We claim this fixed $U$ has the size of $\partial_{H} U$ bounded by $2\sqrt{2} \Delta\sqrt{m+x}$.  Here we emphasize we are taking the boundary of $U$ in $H$ and not $H'$.  Because $U$ is a union of components of $H'$, this means every edge of $\partial_{H} U$ is an edge of $H$ and not an edge of $H'$.  By definition of $H'$, it follows every edge of $\partial_{H} U$ was a deleted edge, of which there are at most $2\sqrt{2} \Delta\sqrt{m+x}$.  Because $U$ is a balanced set, by definition of edge bisection width, $\bis(H) \leq 2\sqrt{2} \Delta\sqrt{m+x}$.  Rearranging in terms of $x = \crn(H)$ gives the desired inequality.  The second assertion follows readily if $\bis(H) \geq \eta m$ and $m \geq 16\Delta^{2}/\eta^{2}$.  
\end{proof}

For a sequence of cyclic permutations $\sigma_{L}$ of length $L$, Lemma \ref{lem:bisection-crossing} provides a quadratic lower bound for $\crn(H_{\sigma})$ in $L$ provided the edge bisection width $\bis(H_{\sigma_{L}})$ grows at least linearly in $L$.  That is to say, if we are able to produce a sequence of permutations of length $L$ for which the associated multigraph has edge bisection width that grows linearly, then we will show the crossing numbers of the associated multigraphs have quadratic growth rate, and thus quadratically growing Jordan crossing numbers by Lemma \ref{lem:cr-lower}.  \\
\\
Our goal is thus to produce permutations $\sigma$ of $[L]$ whose associated cycle-pair graphs $G_{\sigma}$ as in Definition \ref{def:cpmg} have large edge bisection width.  For a fixed $L$, one could look at each individual graph $G_{\sigma}$ for the $L!$ choices of an ordering of $[L]$, and see which of these graphs have sufficiently large edge bisection width.  If for each $L$, we are able to a find a suitable one, we choose it, and this would define our desired sequence.  Constructively showing the existence of these permutations has proven difficult, so instead, our proofs rely on probabilistic methods.  A bit more specifically, we count the number of orderings which fail to satisfy the sufficiently large edge bisection width property, and prove that you can always find a `good' permutation in the complement.  The probabilistic aspect comes choosing a random permutation and we show that for sufficiently large $L$, a generic permutation will satisfy our desired property.  This is the content of Theorem \ref{thm:random-path}.  
\\
\\
Because many of our arguments rely on counting principals, we introduce the so-called \emph{binary entropy function} defined below.
\begin{equation}\label{eq:binent}
h(t) = -t\log(t) - (1-t)\log(1-t) 
\end{equation}
This function has been studied thoroughly and makes its appearance in several different fields of mathematics, particularly Information Theory.  Its primary use appears to be approximations of binomial expansions and coefficients.  For our purposes, we will simply need some inequalities involving the binomial coefficients and the binary entropy function.  The first inequality we will need is standard and the proof is omitted but can be found in Example 11.1.3 of Cover and Thomas \cite{cover2006elements}.  These bounds hold for all integers satisfying $0 \leq k \leq N$.  
\begin{equation}\label{eq:binest}
\frac{1}{N+1}\exp\left(Nh(k/N)\right) \leq \binom{N}{k} \leq \exp\left(Nh(k/N)\right)
\end{equation}

The next Lemma will seem opaque, but will later be used for estimates on probabilities involving the binary entropy function.  For the ease of book keeping, we consolidate them here in a single lemma.   
\begin{lemma}\label{lem:binent}
There exists an $\eta > 0$ and an integer $L_{0}$, dependent on the choice of $\eta$, so that for all $L > L_{0}$, the inequalities of Equation \ref{eq:linequal} below hold. 
\begin{equation}\label{eq:linequal}
4\left(\eta L + 3\right)^{2}\exp\left(2(L+1)h(3\eta)\right) \leq \exp\left(\frac{3}{4}h(1/3)L\right) \phantom{==} \log(L+1) \leq \frac{1}{16}h(1/3)L \phantom{==} \eta L \geq 4-3\eta
\end{equation} 
\end{lemma}

\begin{proof}
Choose an $\eta > 0$ so small that $\eta < 1/6$ and $2h(3\eta) < \frac{1}{2}h(1/3)$.  The existence of such an $\eta$ follows because $h(t)$ tends to $0$ as $t$ tends to zero from the right.  Fix such an $\eta$.  We now need to establish the existence of $L_{0}$.  The existence of $L_{0}$ for the second and third inequalities is clear.  Thus we only need to justify the existence of $L_{0}$ for the first inequality.  To this end, we investigate the ratio below
\begin{equation*}
\frac{4(\eta L + 3)^{2}\exp\left(2(L+1)h(3\eta)\right)}{\exp\left(\frac{3}{4}h(1/3)L\right)} = C\left(\frac{(\eta L+3)^{2}}{\exp({\frac{1}{4}h(1/3)L)}}\right)\left(\frac{\exp(2h(3\eta)L)}{\exp{(\frac{2}{4}h(1/3)L)}}\right)
\end{equation*} 
In the above equality the $C$ has absorbed the constant terms of the left hand side ratio.  As $L \to \infty$, the first parenthetical term on the right clearly goes to $0$ as the denominator has positive exponential growth in $L$ whereas the numerator is polynomial in $L$.  The second parenthetical term on the right also goes to $0$ because by hypothesis, $2h(3\eta) < \frac{1}{2}h(1/3)$.  Thus the ratio must be less than $1$ for all $L$ bigger than some sufficiently large $L_{0}$ which proves our claim.   
\end{proof}

Having established this, we introduce some combinatorial terminology and notation.  Let $\sigma$ be a fixed linear ordering of $[L]$ denoted by $[L]_{\sigma}$.  We define an interval in $[L]_{\sigma}$ in the typical sense of a linear ordering, and emphasize that we permit singletons as intervals.  Due to the fact that we permit singletons as intervals, every subset $U \subset [L]_{\sigma}$ decomposes trivially as the union of disjoint intervals.  We need therefore introduce a more `coarse' decomposition which we refer to as the \emph{separated union}.  
\begin{definition}[Separated Sets]\label{def:separated_int}
Let $[L]_{\sigma}$ be a linear ordering of $[L]$.  We say two sets $A,B \subset [L]_{\sigma}$ are separated if they are disjoint and there exists a $c \in [L]$ for which $A < c < B$ or $B < c < A$.  We say the element $c$ \emph{separates} $A$ and $B$.  We say a collection of subsets $\{A_{i}\}_{i=1}^{i=m}$ of $[L]$ is separated if the collection is pairwise separated.
\end{definition}

For the sake of clarity we provide some examples.  Consider $[8]$ under the standard ordering $[8]_{0}$.  The interval $I_{1} = \{1,2,3\}$ and the interval $I_{2} = \{4,5,6\}$ are disjoint intervals but not separated intervals as no element exists between $3$ and $4$ in $[8]_{0}$.  The intervals $J_{1} = \{1,2,3\}$ and $J_{2} = \{5,6,7\}$ are are separated as $J_{1} < 4 < J_{2}$.  Next we prove a lemma regarding the decomposition mentioned before Definition \ref{def:separated_int}.  

\begin{lemma}[Separated Interval Decomposition]\label{lem:sep_int_decom}
Let $\sigma$ be a linear ordering on $[L]$.  Then every non-empty subset $U \subset [L]_{\sigma}$ admits a separated interval decomposition as $U = \bigcup_{i=1}^{i=k} I_{i}$ where the collection of intervals $\{I_{i}\}_{i=1}^{i=k}$ is separated.  If we label the intervals to be increasing in order, $I_{1} < I_{2} < \hdots < I_{k}$, then this decomposition of $U$ is unique.
\end{lemma}

\begin{proof}
Order the elements of $[L]$ in a line according to $\sigma$.  Assign any subset $U \subset [L]_{\sigma}$ the binary string of $0$'s and $1$'s $w := x_{0}x_{1}\hdots x_{L-1}$ defined by $x_{j} := 1$ if $\sigma_{j} \in U$ and $0$ if $\sigma_{j} \notin U$.  Reading left to right, let $I_{1}$ be the first contiguous string of $\sigma_{j}$'s where $x_{j} = 1$ in the word $w$.  Let $I_{2}$ be the second continuous string of $\sigma_{j}$'s where $x_{j} = 1$ in the word $w$, so on and so forth.  There are finitely many because $[L]$ is finite, and this provides ours desired decomposition.  It is clear from construction that $I_{1} < I_{2} < \hdots < I_{k}$ and this decomposition is unique.  
\end{proof}

We remark in the case where $U \subset [L]_{\sigma}$ is the empty set, we simply define its separated interval decomposition to be the collection containing just the empty set.  With this established, we introduce a collection of subsets that will repeatedly appear throughout our counting arguments.  By Lemma \ref{lem:sep_int_decom}, every subset of $[L]_{\sigma}$ admits a unique separated interval decomposition, and thus we assign each subset of $[L]_{\sigma}$ the invariant which is the number of separated intervals in its separated interval decomposition.  In some sense, this is a measure of how connected $U$ is inside of $[L]_{\sigma}$.  For $0 \leq r \leq s \leq L$, define $I_{s}(\sigma,r)$ to be the collection of all $s$-element subsets of $[L]_{\sigma}$ whose separated interval decomposition contains at most $r$-intervals.  By our definitions, for either $s = 0$ or $r = 0$, $I_{s}(\sigma,r)$ is just the collection containing just the empty set.  
\\
\\
With these definitions established, we give an overview of the remainder of this section.  In Lemma \ref{lem:sepintbounds}, we provide counting arguments to show that for fixed integers $0 \leq r \leq s$ and $L$, if one picks a random linear ordering for $[L]$, and a fixed subset $U \in [L]$, the probability that $U$ is in \emph{both} $I_{s}(\sigma,r)$ and $I_{s}(\sigma_{0},r)$ is bounded above by a ratio of terms involving $I_{s}(\sigma_{0},r)$ and binomial coefficients.  In Lemma \ref{lem:estimate_prob}, we bound this ratio by a decreasing exponential once we restrict our attention to balanced sets $U$ with $|U| = s \in [L/3,2L/3]$, $r = \eta L$, and sufficiently large $L$.  In Lemma \ref{lem:badbound}, we show that every set which fails to have large edge bisection width must have a small number of intervals in its separated interval decomposition in both $[L]$ and $[L]_{\sigma}$.  In Theorem \ref{thm:random-path}, we combine these results to show that a randomly chosen permutation of $[L]$ has an exponentially small likelihood of containing balanced subsets whose separated interval decomposition in both $[L]$ and $[L]_{\sigma}$ has a small amount of separated intervals.  This last remark follows more or less from the following heuristic.  Say you take a balanced subset $U$ in $[L]_{0}$ whose separated interval decomposition has few intervals.  Then apply a random permutation to it.  Odds become very high that you destroyed the connectivity of the intervals in this process if $L/3 \leq |U| \leq 2L/3$, so its new separated interval decomposition has more separated intervals.  Take for example the case $U$ is just an interval of length $L/3$.  After applying a random permutation, it will almost certainly become a bunch of intervals in $[L]_{0}$.  This is the key insight that helps prove our result.  To this end, we begin with Lemma \ref{lem:sepintbounds}.  

\begin{lemma}[Separated Intervals Bounds]\label{lem:sepintbounds}
Let $r \leq s$ be non-negative integers and $L$ be a positive integer.  Pick some $\sigma$ randomly amongst the $L!$ linear orderings of $[L]$ under the uniform distribution.  Let $I_{s}(\sigma,r)$ and $I_{s}(\sigma_{0},r)$ denote the collections of $s$-elements subsets of $[L]$ which admit separated interval decompositions of at most $r$-intervals in the orderings $[L]_{\sigma}$ and $[L]_{0}$ respectively.  Define the random variable $X_{s,r}(\sigma) := |I_{s}(\sigma,r) \cap I_{s}(\sigma_{0},r)|$.  Then the probability that $X_{s,r}$ is greater than $0$ is bounded above $|I_{s}(\sigma_{0},r)|^{2}/\binom{L}{s}$.  Symbolically $\Fpr\left(X_{s,r} \geq 1 \right) \leq |I_{s}(\sigma_{0},r)|^{2}/\binom{L}{s}$.  
\end{lemma}

We remark that the usefulness of the probability bound above depends on the values of $s$, $r$, and $L$ chosen.  For our purposes we will later restrict our parameters and make these estimates more useful, however, for the purpose of clarity we have isolated the counting parts of our argument to this lemma.  

\begin{proof}
Pick a $r$ and $s$ as in the hypotheses, and a random linear ordering $\sigma$ of $\{0,1,\hdots, L-1\}$.  We calculate the probability that $U$ is inside $I_{s}(\sigma,r)$.  Note that $\sigma$ induces an order preserving bijection between $[L]$ and $[L]_{\sigma}$ which takes $i$ to $\sigma_{i}$.  Because $i < j$ in $[L]$, if and only if $\sigma(i) < \sigma(j)$ in $[L]_{\sigma}$, $\sigma$ must take intervals to intervals and therefore induces a bijection between $I_{s}(\sigma,r)$ and $I_{s}(\sigma_{0},r)$.  Specifically, $U \in I_{s}(\sigma_{0},r)$ if and only if $\sigma(U) \in I_{s}(\sigma,r)$.  
\\
\\
As a consequence, a fixed $U \in I_{s}(\sigma_{0},r)$ is also inside $I_{s}(\sigma,r)$ if and only if $\sigma^{-1}(U) \in I_{s}(\sigma_{0},r)$.  The possible values of $\sigma^{-1}(U)$ as $\sigma$ ranges through all possible linear orderings of ${0,1,\hdots, L-1}$ are uniformly distributed amongst all $s$-element subsets of $\{0,1,\hdots, L-1\}$.  Because $\sigma^{-1}(U)$ is equally likely to be any $s$-element subset of $[L]$, the probability that $U$ is also in $I_{s}(\sigma,r)$ for a randomly chose linear ordering $\sigma$ is equal to $|I_{s}(\sigma_{0},r)|/\binom{L}{s}$. \\
\\
Having established this probability for a fixed $U \in I_{s}(\sigma_{0},r)$, we calculate the expected value of the number of all such $U$ in $I_{s}(\sigma_{0},r)$ which are also in $I_{s}(\sigma,r)$ for a randomly selected linear ordering $\sigma$.  For each $U \in I_{s}(\sigma_{0},r)$, define the random variable $Y_{U,r}(\sigma)$ to be $+1$ if $U \in I_{s}(\sigma,r)$ and $0$ otherwise.  We emphasize here the sample space is the collection of all possible linear orderings of $\{0,1,\hdots, L-1\}$.  By our previous calculation, for a each $U \in I_{s}(\sigma_{0},r)$, $\Fpr(Y_{U,r} = 1) = |I_{s}(\sigma_{0},r)|/\binom{L}{s}$.  Define the random variable $X_{s,r} := \sum_{U \in I_{s}(\sigma_{0},r)} Y_{U,r}$.  By definition, the value of $X_{s,r}$ for a linear ordering $\sigma$ equals the size of the intersection of $I_{s}(\sigma,r)\cap I_{s}(\sigma_{0},r)$ .  We aim to bound the probability that $X_{s,r} \geq 1$.  \\
\\
To this end, by Markov's inequality we have that $\Fpr (X_{s,r} \geq 1) \leq \mathbb{E}[X_{s,r}]$ so it suffices to bound the expected value.  By linearity of the expected value 
\begin{equation*}
\mathbb{E}[X_{s,r}] = \sum_{U \in I_{s}(\sigma_{0},r)} \mathbb{E}[Y_{U,r}] = \sum_{U \in I_{s}(\sigma_{0},r)} (+1)\Fpr(Y_{U,r} = 1) = \frac{|I_{s}(\sigma_{0},r)|^{2}}{\binom{L}{s}}
\end{equation*}
Thus for a randomly chosen linear ordering $\sigma$ of $[L]$, the probability that there exists a set $U$ which is expressible as the union of at most $r$-separated intervals in both the natural ordering $[L]_{0}$ and the random ordering $[L]_{\sigma}$ is bounded above by $|I_{s}(\sigma_{0},r)|^{2}/\binom{L}{s}$.
\end{proof}

The next lemma addresses the asymptotics of the ratio of $|I_{s}(\sigma_{0},r)|^{2}$ and $\binom{L}{s}$ and shows that the probability there exists a balanced subset $U \subset [L]$ with a small interval number of intervals in the separated interval decomposition in both $[L]_{0}$ and $[L]_{\sigma}$ for a random permutation $\sigma$ is nearly trivial.  Again the intuition for this is that the permutation will generically take an interval and separate it into a bunch of other intervals.

\begin{lemma}[Uniformly Bounded Probabilities]\label{lem:estimate_prob}
There exists a real $\eta > 0$ and positive integer $L_{0}$, which depends on $\eta$, such that the following holds.  We can find an $a > 0$ so that for any integer $L > L_{0}$, and all integers $s \in [L/3,2L/3]$, if we let $r = \lceil \eta L \rceil + 1$, then the probability that $X_{r,s}$ is bigger than $0$ is bounded above by $e^{-aL}$. 
\end{lemma}

\begin{proof}
Choose and $0 < \eta < 1/6$ satisfying $2h(3\eta) < \frac{1}{2}h(1/3)$ as in $\eta$ Lemma \ref{lem:binent}.  Now pick a positive integer $L_{0}$ as guaranteed by Lemma \ref{lem:binent} so that all three inequalities of Equation \ref{eq:linequal} hold for all $L > L_{0}$.  We now aim to bound the size of $I_{s}(\sigma_{0},r)$ for all $s \in [L/3, 2L/3]$. 
 Let $U$ be an $s$-element subset of $[L]$ expressible as a union of at most $r$ separated intervals in the order $[L]_{0}$.  Say there are $k$ such intervals.  Each interval $I_{i}$ may be assigned the pair of elements $(a_{i},b_{i})$ where $a_{i}$ is the smallest element of $I_{i}$, and $b_{i}$ is the immediate successor of the largest element of $I_{i}$.  To make this well defined, we define the immediate successor of $L-1$ to be $L \notin [L]$.  By definition $a_{i} < b_{i}$ for each $1 \leq i \leq k$.  Thus each such $U$ can be assigned the $k$-pairs of elements $\{(a_{i},b_{i})\}_{i=1}^{k}$ from $[L]\cup\{L\}$ which uniquely determine $U$, and its separated interval decomposition.  By hypothesis $k$ can be any number between $0$ and $r$, thus we have the bounds below.  
\begin{equation}\label{eq:choose_est}
|I_{s}(\sigma_{0},r)| \leq \sum_{k = 0}^{r} \binom{L+1}{2k} \leq \sum_{k=0}^{2r} \binom{L+1}{k}
\end{equation}
While this is a significant overestimate, it will be easier to work with than being precise.  In particular, we use the estimates in Equation \ref{eq:binest} to yield the inequality below for each $0 \leq k \leq 2r$.  
\begin{equation*}
\binom{L+1}{k} \leq \exp\left((L+1)h\left(\frac{k}{L+1}\right)\right)
\end{equation*}
By the third inequality of Equation \ref{eq:linequal}, $2r \leq 3\eta(L+1)$ and by definition of $\eta$, $3\eta < 1/2$.  Thus the monotonicity of $h$ on $[0,1/2]$ yields for each $0 \leq k \leq 2r$ that $h(k/(L+1)) \leq h(3\eta)$.  Combining these inequalities with Equation \ref{eq:choose_est}, we have for each $s \in [L/3,2L/3]$,
\begin{equation*}
|I_{s}(\sigma_{0},r)| \leq \sum_{k=0}^{2r} \binom{L+1}{k} \leq (2r+1)\exp\left((L+1)h(3\eta)\right)
\end{equation*}
By squaring both sides of the above inequality, and using the fact that $L$ satisfies the first inequality of Equation \ref{eq:linequal}, we obtain for each $s \in [L/3,2L/3]$ the inequality below.
\begin{equation}\label{eq:numbound}
|I_{s}(\sigma_{0},r)|^{2} \leq \exp\left(\frac{3}{4}h(1/3)L\right)
\end{equation}
Similarly, Equation \ref{eq:binest} and monotonicity of $h$ gives for each $s \in [L/3,2L/3]$ the inequality below.
\begin{equation}\label{eq:denbound}
\frac{1}{L+1}\exp\left(h(1/3)L\right) \leq \binom{L}{s}  
\end{equation}
Consequently for each $s \in [L/3,2L/3]$ we have that
\begin{equation*}
 \frac{|I_{s}(\sigma_{0},r)|^{2}}{\binom{L}{s}} \leq \left(L+1\right)\exp\left(-\frac{1}{4}h(1/3)L\right) \leq \exp\left(-\frac{3}{16}h(1/3)L\right)
\end{equation*}
The last inequality above is a consequence of the third inequality of Equation \ref{eq:linequal}.  Combining this result with Lemma \ref{lem:sepintbounds} establishes our result.
\end{proof}

We wish to show that the probability a randomly selected linear ordering has a sufficiently large edge bisection width.  The next lemma relates this property to the separated interval decomposition in both $[L]_{0}$ and $[L]_{\sigma}$ for a randomly selected order.  Before we state this, we need to introduce some terminology.  \\
\\
Recall from Definition \ref{def:ppmg}, for a given permutation $\sigma$ of $[L]$, we define the path-pair graph $H_{\sigma}$ to be the union of the Hamiltonian path graphs induced by $[L]_{0}$ and $[L]_{\sigma}$.  We wish to count the number of balanced subsets $U$ of $[L]$ which have the property that $|\partial_{H_{\sigma}} U| < \eta L$.  We call such sets $U$ \emph{bad}.  Due to the multigraph structure of $H_{\sigma}$ coming from two linear orderings, one from $[L]$ and one from $[L]_{\sigma}$, the set of boundary edges $\partial_{U}(H_{\sigma})$ is the union of the boundary edges of $U$ coming from the two Hamiltonian paths associated to $[L]_{0}$ and $[L]_{\sigma}$.  Counting these boundary edges can be done in terms of the linear ordering induced on $[L]$ by both $[L]_{0}$ and $[L]_{\sigma}$.  A bit more precisely, we show every bad set $U$ has a separated interval decomposition with at most a fixed number of intervals in both $[L]_{0}$ and $[L]_{\sigma}$.  This will allow us to bound the number of bad sets in $[L]$, and apply our previous probability estimates which are exponentially decreasing.  We state this in the lemma below.

\begin{lemma}[Bad Set Bound]\label{lem:badbound}
Let $\eta$ and $r$ be as in Lemma \ref{lem:estimate_prob} and $\sigma$ be a permutation of $[L]$.  Form the path-pair graph $H_{\sigma}$.  Every bad balanced subset $U \subset [L]$, i.e. a subset of $[L]$ with $L/3 \leq |U| \leq 2L/3$ satisfying $|\partial_{H_{\sigma}} U| < \eta L$, is the union of at most $r = \lceil \eta L \rceil + 1$ separated intervals in both the order $[L]_{0}$ and $[L]_{\sigma}$.  
\end{lemma}

\begin{proof}
Let $U$ be bad as in the hypotheses for some fixed ordering of $[L]_{\sigma}$.  If we define $P_{0}$ and $P_{\sigma}$ to be the Hamiltonian paths associated to the orderings $[L]_{0}$ and $[L]_{\sigma}$ of $[L]$ respectively, then it is clear that $H_{\sigma} = P_{0}\cup P_{\sigma}$ as a multigraph.  The number of edges leaving the subset $U$ in $H_{\sigma}$ is thus the sum of the number of edges leaving $U$ in $P_{0}$ and the number of edges leaving $U$ in $P_{\sigma}$.  Symbolically that is to say $|\partial_{H_{\sigma}} U| = |\partial_{P_{0}}U| + |\partial_{P_{\sigma}}U|$.  Consequently if $U$ is bad in $H_{\sigma}$, then both summands $|\partial_{P_{0}}U|$ and $|\partial_{P_{\sigma}}U|$ are also less than $\eta L$.  We focus our attention for now on $P_{0}$.  \\
\\
Let $b := |\partial_{P_{0}}U|$.  Let $k$ be the number of intervals in the separated interval decomposition of $U$ inside $[L]_{0}$.  We claim $k$ is at most $b+1$.  To each of the $k$-intervals we assign the edge $\{b_{i},b_{i}+1\} \in \partial_{P_{0}} U$ where here $b_{i}$ is the maximal element of the interval and $b_{i}+1$ is its immediate successor in $[L]_{0}$.  To make this well defined, and if $b_{i} = L-1$, we add the additional edge $\{L-1,L\}$.  Because each interval is separated, no two distinct intervals can map to the same edge in $\partial_{P_{0}} U \cup \{\{L-1,L\}\}$.  Thus, the number of possible intervals is at most $b+1$.  By hypothesis, $b < \eta L$, therefore $b+1 < \lceil \eta L \rceil + 1 = r$.  The same argument shows that $U$ is also a union of at most $r$-separated intervals in $[L]_{\sigma}$, thus proving our claim.
\end{proof}

Lemma \ref{lem:badbound} says that for a fixed permutation $\sigma$, if $U$ is a balanced bad set, then $U \in I_{s}(\sigma_{0},r) \cap I_{s}(\sigma,r)$.  Lemmas \ref{lem:estimate_prob} and \ref{lem:sepintbounds} tell us that if we choose a random permutation $\sigma$, the size of this intersection decreases exponentially in $L$, and this rate of exponential decrease is uniformly bounded for each choice of $s \in [L/3,2L/3]$.  Combining these observations gives us the theorem below which says there exists sufficiently large $L$ and cyclic permutations $\sigma$ of $[L]$ for which the edge bisection width of the path-pair graph $H_{\sigma}$ has edge bisection width that grows linearly in $L$, and, these permutations are generic.  Motivated by our construction in Lemma \ref{lem:tile}, we call these \emph{crossing-rich blocks}.
%
\begin{theorem}[Random path pairs]\label{thm:random-path}
There exists a real $\eta > 0$ and positive integer $L_{0}$, which depends on $\eta$, such that the following holds.  We can find a $b > 0$ so that for any integer $L > L_{0}$, and all integers $s \in [L/3,2L/3]$, if we let $r = \lceil \eta L \rceil + 1$, then the probability that a bad balanced set $U$ exists for the path-pair multigraph $H_{\sigma}$ where $\sigma$ is some uniformly random selected ordering of $[L]$ is bounded above by $e^{-bL}$.  Symbolically, 
\begin{equation*}
\Fpr\left(\text{ There exists a }U \subset [L] \text{ with } L/3 \leq |U| \leq 2L/3 \text{ and } |\partial_{H_{\sigma}} U| < \eta L\right) < e^{-bL}
\end{equation*}
In particular, this means we can find a fixed $\eta > 0$ and a constant $c_{\eta} > 0$, so that for all sufficiently large $L$, there exists sequences of permutations $\sigma_{L}$ of $[L]$ whose associated pair-path multigraphs $H_{\sigma_{L}}$ have $\crn(H_{\sigma_{L}}) \geq c_{\eta} L^{2}$.  Moreover, if one chooses a random permutation of $[L]$, $H_{\sigma_{L}}$ generically has $\crn(H_{\sigma_{L}}) \geq c_{\eta} L^{2}$ for sufficiently large $L$.    
\end{theorem}

\begin{proof}
Let $\eta$ and $L > L_{0}$ be as in Lemma \ref{lem:estimate_prob}.  Let $c_{\eta} = 16^{2}/\eta^{2}$, and in addition, ensure that $L > c_{\eta}$.  Pick a random linear ordering $\sigma$ of $[L]$ from the $L!$ choices.  Let $s$ be an integer in $[L/3,2L/3]$, and $Z_{s,\eta}$ be the random variable which counts the number of bad balanced subsets of $H_{\sigma}$ of size $s$.  By Lemma \ref{lem:badbound}, $Z_{s,\eta}(\sigma) \leq X_{s,r}(\sigma)$ for each $s$ because every bad balanced subset of $H_{\sigma}$ of size $s$ is the separated union of at most $r$-intervals in both $[L]_{0}$ and $[L]_{\sigma}$.  By Lemma \ref{lem:estimate_prob}, $\Fpr(X_{s,r} > 0) \leq e^{-aL}$ for each $s \in [L/3,2L/3]$ where $a = (3/16) h(1/3)$.  If we define $Z_{\eta}$ to be the sum over all $Z_{s,\eta}$ for each $s \in [L/3,2L/3]$ which counts the number of bad subsets of $H_{\sigma}$, then we have the following inequalities below.
\begin{equation*}
\Fpr\left(Z_{\eta} > 0\right) = \sum_{s \in [L/3,2L/3]} \Fpr\left(Z_{s,\eta} > 0 \right) \leq \sum_{s \in [L/3,2L/3]} \Fpr\left(X_{s,r} > 0 \right) \leq Le^{-aL} \leq e^{-bL}
\end{equation*}
The last inequality follows by choice of $L$ which satisfies the second inequality of Equation \ref{eq:linequal}.  Thus the probability there exists a bad subset in $H_{\sigma}$ is exponentially decreasing as claimed.  \\
\\
In particular, this means that a randomly selected linear ordering $\sigma_{L}$ on $[L]$ for sufficiently large $L$ generically induces a path-pair graph $H_{\sigma_{L}}$ whose edge bisection width is at least $\eta L$.  By Lemma \ref{lem:bisection-crossing}, $\crn(H_{\sigma_{L}}) \geq c_{\eta} L^{2}$ where we choose $\Delta = 4$ because $H_{\sigma_{L}}$ is the union of two Hamiltonian paths.  As the inequalities hold for a fixed $\eta$ and all $L > \max\{L_{0},c_{\eta}\}$, this produces the desired sequence of such permutations $\sigma_{L}$.  
\end{proof}

\section{Assembling Growth Rates}
\label{sec:growthrates}

This section combines the upper and lower bounds from Sections \ref{sec:upper-bound} and \ref{sec:lower-bound} to realize growth rates for meandering numbers of permutations between linear and quadratic.  The method is as follows.  Choose a block length function $L=L(n)$ satisfying $1 \leq L(n) \leq n$ with $L(n)$ tending to infinity and tile $[n]$ with crossing-rich tiles of length $L(n)$ from Theorem \ref{thm:random-path} so that the upper bound of Section~\ref{sec:upper-bound} and the lower bound of Section~\ref{sec:lower-bound} pin the crossing number to $\Theta(nL(n))$ from both sides.  This is the content of Theorem~\ref{thm:intermediate}, which produces cyclic permutations $\sigma_{n}$ of $[n]$ with meandering number $\mu(\sigma_{n}) = n+\Theta(nL(n))$ for every prescribed scale $L(n)$.  We then read off four specializations as in Corollary~\ref{cor:natural-families} which records the families $\Theta(n)$, $\Theta(n\log n)$, $\Theta(n^{1+\alpha})$, and $\Theta(n^2)$ as meandering numbers.  We conclude this work with Corollary~\ref{cor:random-global} which proves Schwartz's quadratic meander number conjecture for sufficiently large $n$.


\begin{theorem}[All intermediate scales]\label{thm:intermediate}
Let $L(n)$ be an integer-valued function with
\begin{equation*}
        L(n)\to\infty,
        \qquad
        1\le L(n)\le n.
\end{equation*}
Then there exist cyclic permutations $\sigma_n$ of $[n]$ such that
\begin{equation*}
        \mu(\sigma_n)=n+\Theta(nL(n)).
\end{equation*}
\end{theorem}

\begin{proof}
Write $L=L(n)$, and let
\begin{equation*}
        q=\left\lfloor \frac{n}{L}\right\rfloor,
        \qquad
        R=n-qL.
\end{equation*}
Thus $0\le R<L$.  For sufficiently large $n$, $L(n) \ge L_0$, where $L_0$ is as in Theorem~\ref{thm:random-path}.  Partition $[n]$ into
$q$ consecutive blocks $I_{j}$ of size $L$, followed by one final remainder block of size $R$ if $R>0$.  Inside each full block, place a translated copy of the crossing-rich order $\sigma_L$ from Theorem~\ref{thm:random-path} which is quadratically bounded below in crossing number $cL^{2}$ for some fixed $c$.  In the remainder block, if present, use the identity order.  Concatenate the blocks in their natural order.  This defines a cyclic permutation $\sigma_n$.   \\
\\
Because $\sigma_{n}$ contains $q$-copies of the crossing-rich tile of $\sigma_{L}$, each of which is quadratically bounded below in crossing number by $c_{\eta}L^{2}$, Proposition~\ref{prop:block} gives
\begin{equation*}
        \crn(G_{\sigma_{n}})\ge q\crn(H_{\sigma_{L}}) \geq qc_{\eta}L^2.
\end{equation*}
By definition of $q$ and $L$, it follows that $qL^{2} \geq nL/2$, thus 
\begin{equation*}
        \crn(G_{\sigma_{n}}) \ge \frac{c_{\eta}}{2} nL.
\end{equation*}
For the upper bound, Proposition~\ref{prop:block-upper} and Lemma \ref{lem:cr-lower} give
\begin{equation*}
        \crn(G_{\sigma_{n}}) \le qL^2+R^2.
\end{equation*}
Since $qL^2\le nL$ and $R^2<L^2\le nL$, we have
\begin{equation*}
        \crn(G_{\sigma_{n}}) \le 2nL.
\end{equation*}
Combining this with the lower bound gives
\begin{equation*}
        \crn(G_{\sigma_{n}}) =\Theta(nL).
\end{equation*}
Finally, Lemma~\ref{lem:exact} gives $\mu(\sigma_n)=n+\Theta(nL(n))$ as claimed.
\end{proof}

In the following corollary, we record several familiar asymptotic growth rates as immediate special cases of Theorem~\ref{thm:intermediate}.

\begin{corollary}\label{cor:natural-families}
There are cyclic permutations on $n$ letters with the following growth rates in Jordan crossing number:
\begin{equation*}
        \Theta(n),\qquad
        \Theta(n\log n),\qquad
        \Theta(n^{1+\alpha})\quad(0<\alpha<1),\qquad
        \Theta(n^2).
\end{equation*}
\end{corollary}

\begin{proof}
The linear family comes from the identity permutation on $n$-letters.  For the remaining families apply Theorem~\ref{thm:intermediate} with $L(n)$ equal to $\lfloor \log n\rfloor, \lfloor n^\alpha\rfloor$, and $n$ respectively where $\alpha$ is some real number strictly between $0$ and $1$ and use the fact that Jordan crossing number equals the crossing number of $G_{\sigma_{n}}$ by Lemma \ref{lem:cr-lower}.  
\end{proof}

Corollary \ref{cor:natural-families} provides families of cyclic permutations with several different growth rates in the meandering number.  The largest of which when $L(n) = n$ follows readily from Theorem \ref{thm:random-path}, and shows there is $c > 0$ so that for sufficiently large $n$, one can always find a linear ordering on $[n]$ that will induce a permutation whose meander number is bounded below by $cn^{2}$.  We can then scale $c$ sufficiently small so that the meander number is bounded below for all $n$ that were not sufficiently large.  This observation combined with Corollary \ref{cor:global-upper} proves Schwartz' Conjecture 3.4 (Meander Number), and hence his Topological Salesman Conjecture as discussed in Section \ref{sec:introduction}, to be true for sufficiently large $n$.  For completeness we state this as a Corollary below.  
\begin{corollary}[Asymptotic Quadratic Meander Number Growth]\label{cor:random-global}
There exists constants $c, C > 0$ such that the probability a randomly selected linear ordering on $[n]$ induces a cyclic permutation $\sigma_{n}$ satisfies 
\begin{equation}\label{eq:thetan2}
        cn^{2} \leq \mu(\sigma_{n}) \leq Cn^{2}
\end{equation}
tends to $1$ in $n$.  In particular, this means there exists permutations satisfying the above inequalities for all $n$, thus $\mu_{n} = \Theta(n^{2})$.    
\end{corollary}
We conclude by remarking our results are non-constructive, however, we are able to provide some explicit bounds.  So long as $\eta$ and $L_{0}$ in Theorem \ref{thm:random-path} are chosen to satisfy the various inequalities of Equation \ref{eq:linequal} and $L_{0} > 16^{2}/\eta^{2}$, the inequalities of Equation \ref{eq:thetan2} hold with $c = 16^{2}/\eta^{2}$, $C = 2$, and $n > L_{0}$.  This shows are results are eventually effective, but, as remarked above we may scale $c$ down so the inequalities hold for all $n$.  We thank Richard Schwartz for pointing this argument out.  \\
\\
We note that the value of $n$ for which the inequalities become effective is quite large.  One would have to verify, at least, the first $16^{2}/(1/6)^{2} = 9,216$ cases as we assume $\eta < 1/6$ in Lemma \ref{lem:binent} so $L > c_{\eta} > 16^{2}/\eta^{2}$.  It would be of great interest to provide an algorithm to construct a family of such permutations $\sigma_{n}$ on $[n]$ with quadratically growing meander number.  


\bibliographystyle{amsalpha}
\bibliography{refs}

\end{document}

%% file: figures/meander-jordan-dictionary.tex
\begin{figure}[ht]
\centering

\begin{subfigure}[c]{0.31\textwidth}
\centering
\resizebox{\linewidth}{!}{%
\begin{tikzpicture}[x=1.15cm,y=1.15cm, every node/.style={font=\scriptsize}]
  \colorlet{Scolor}{green!50!black}
  \colorlet{Sfill}{green!70!black}

  \tikzset{
    directed arc/.style={
      very thick,
      dashed,
      postaction={decorate},
      decoration={
        markings,
        mark=at position 0.5 with {\arrow{stealth}}
      }
    }
  }

  \path[use as bounding box] (-.45,-1.95) rectangle (5.45,2.45);

\draw[line width=1.4pt] (-.45,0) -- (5.45,0);
\draw[line width=1.4pt,->,>=stealth] (2.25,0) -- (2.75,0);

  \draw[directed arc] (0,0) arc[start angle=180,end angle=0,radius=1.5];
  \draw[directed arc] (4,0) arc[start angle=180,end angle=0,radius=.5];
  \draw[directed arc] (2,0) arc[start angle=0,end angle=180,radius=.5];

  \draw[directed arc] (3,0) arc[start angle=180,end angle=360,radius=.5];
  \draw[directed arc] (5,0) arc[start angle=360,end angle=180,radius=1.5];
  \draw[directed arc] (1,0) arc[start angle=360,end angle=180,radius=.5];

  \foreach \x/\lab in {2/2,3/3}{
    \node[fill,circle,minimum size=6pt,inner sep=0pt] at (\x,0) {};
    \node[below=10pt] at (\x,0) {$\lab$};
  }

  \foreach \x/\lab in {0/0,1/1,4/4,5/5}{
    \node[
      draw=Scolor,
      fill=Sfill,
      thick,
      rectangle,
      minimum size=6pt,
      inner sep=0pt
    ] at (\x,0) {};
    \node[below=10pt,text=Scolor] at (\x,0) {$\lab$};
  }
\end{tikzpicture}%
}
\caption{}
\end{subfigure}
\hfill
\begin{subfigure}[c]{0.31\textwidth}
\centering
\resizebox{\linewidth}{!}{%
\begin{tikzpicture}[x=1.05cm,y=1.05cm, every node/.style={font=\scriptsize}]
  \colorlet{Scolor}{green!50!black}
  \colorlet{Sfill}{green!70!black}

  \tikzset{
    directed arc/.style={
      very thick,
      dashed,
      postaction={decorate},
      decoration={
        markings,
        mark=at position 0.5 with {\arrow{stealth}}
      }
    }
  }

  \draw[very thick]
    (-.45,0) -- (5.45,0) -- (5.45,2.15) -- (-.45,2.15) -- cycle;

  \draw[very thick,<-,>=stealth] (2.15,2.15) -- (2.85,2.15);

  \draw[directed arc] (0,0) -- (0,1.5) -- (3,1.5) -- (3,0);
  \draw[directed arc] (4,0) -- (4,.5) -- (5,.5) -- (5,0);
  \draw[directed arc] (2,0) -- (2,.5) -- (1,.5) -- (1,0);

  \draw[directed arc] (3,0) -- (3,-.5) -- (4,-.5) -- (4,0);
  \draw[directed arc] (5,0) -- (5,-1.5) -- (2,-1.5) -- (2,0);
  \draw[directed arc] (1,0) -- (1,-.5) -- (0,-.5) -- (0,0);

  \foreach \x/\lab in {2/2,3/3}{
    \node[
      fill,
      circle,
      minimum size=6pt,
      inner sep=0pt
    ] at (\x,0) {};
    \node[below=8pt] at (\x,0) {$\lab$};
  }

  \foreach \x/\lab in {0/0,1/1,4/4,5/5}{
    \node[
      draw=Scolor,
      fill=Sfill,
      thick,
      rectangle,
      minimum size=6pt,
      inner sep=0pt
    ] at (\x,0) {};
    \node[below=8pt,text=Scolor] at (\x,0) {$\lab$};
  }

\end{tikzpicture}%
}
\caption{}
\end{subfigure}
\hfill
\begin{subfigure}[c]{0.25\textwidth}
\centering
\resizebox{.52\linewidth}{!}{%
\begin{tikzpicture}[scale=.80, every node/.style={font=\scriptsize}]
  \colorlet{Scolor}{green!50!black}
  \colorlet{Sfill}{green!70!black}

  \draw[very thick]
    (0,4.2) -- (0,1) -- (.5,1)
    -- (3,1) -- (3,2) -- (.5,2)
    -- (.5,3) -- (2.5,3) -- (3,3) -- (3,4.2)
    -- (3,4.75) -- (0,4.75) -- cycle;

  \draw[very thick,dashed]
    (1,1) -- (1,-.4) -- (0,-.4) -- (0,-2.15)
    -- (3,-2.15) -- (3,-.4) -- (2,-.4) -- (2,1)
    -- (2,3.5) -- (1,3.5) -- (1,1);

  \draw[very thick,->,>=stealth] (1.75,4.75) -- (1.25,4.75);
  \draw[very thick,dashed,->,>=stealth] (1.75,-2.15) -- (1.25,-2.15);

  \foreach \p/\lab/\pos in {
    {(1,1)}/0/below left,
    {(2,1)}/1/below right,
    {(1,3)}/4/above left,
    {(2,3)}/5/above right}{
      \node[
        draw=Scolor,
        fill=Sfill,
        thick,
        rectangle,
        minimum size=6pt,
        inner sep=0pt
      ] at \p {};
      \node[\pos,text=Scolor] at \p {$\lab$};
    }
  \foreach \p/\lab/\pos in {
    {(2,2)}/2/right,
    {(1,2)}/3/left}{
      \node[
        fill,
        circle,
        minimum size=6pt,
        inner sep=0pt
      ] at \p {};
      \node[\pos,yshift=-5pt] at \p {$\lab$};
    }
\end{tikzpicture}%
}
\caption{}
\end{subfigure}

\caption{
Three equivalent realizations of the permutation $(0, 2, 3, 1)$ as the first return map of a meanderic permutation:
\textup{(a)} with a semicircular meander diagram,
\textup{(b)} with a PLJordan realization (Section~\ref{sec:pl-jordan}), and
\textup{(c)} with a permutation tile (Section~\ref{sec:lower-bound}).
The meandric permutation
$(0,3,4,5,2,1)$ has first-return $(0,4,5,1)$ on
$S=\{0,1,4,5\}$, which becomes $(0,2,3,1)$ after relabeling $S$ in its
inherited order.}
\label{fig:mjd}
\end{figure}
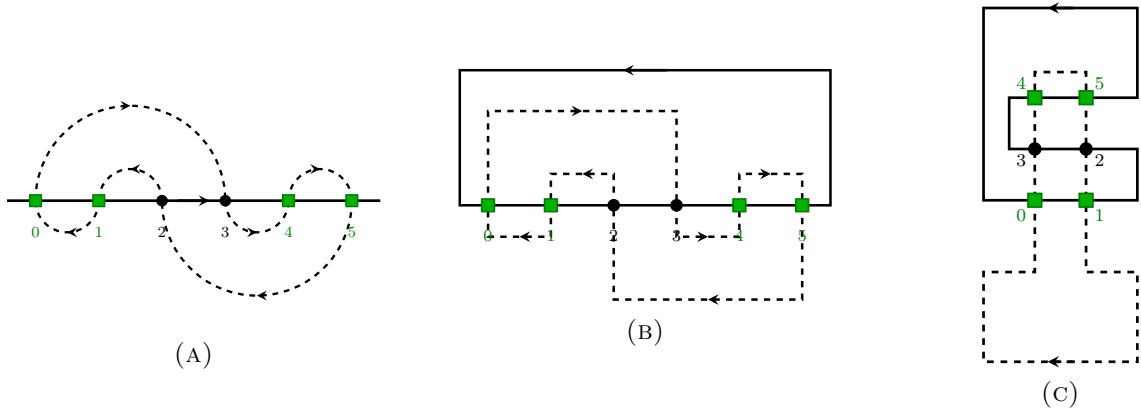

%% file: figures/permutation-tile.tex
\begin{figure}[ht]
\centering
\begin{tikzpicture}[scale=.65, every node/.style={font=\scriptsize}]
  \colorlet{Scolor}{green!50!black}
  \colorlet{Sfill}{green!70!black}

  \tikzset{
    crossinglabel/.style={
      anchor=south west,
      xshift=3pt,
      yshift=2pt,
      inner sep=0pt,
      font=\tiny
    }
  }

  \draw[rounded corners] (-1,-2) rectangle (6,6);

  \draw[very thick]
    (0,6) -- (0,1) -- (.5,1)
    -- (4.5,1) -- (4.5,2) -- (.5,2) -- (.5,3) -- (4.5,3)
    -- (5,3) -- (5,6);

  \draw[very thick,dashed]
    (0,-2) -- (0,-1) -- (1,-1) -- (1,.5)
    -- (1,3.5) -- (2,3.5) -- (2,.5) -- (3,.5)
    -- (3,3.5) -- (4,3.5) -- (4,.5)
    -- (4,-1) -- (5,-1) -- (5,-2);

  \draw[very thick,->,>=stealth] (2.75,2) -- (2.25,2);
  \draw[very thick,dashed,->,>=stealth] (2.25,.5) -- (2.75,.5);

  \foreach \x/\y/\lab in {
    1/1/0,
    2/1/1,
    4/1/3,
    4/2/4,
    3/2/5,
    2/2/6,
    1/3/8,
    3/3/10,
    4/3/11}{
    \node[
      fill,
      circle,
      minimum size=6pt,
      inner sep=0pt
    ] at (\x,\y) {};
    \node[crossinglabel] at (\x,\y) {$\lab$};
  }

  \foreach \x/\y/\lab in {
    3/1/2,
    1/2/7,
    2/3/9}{
    \node[
      draw=Scolor,
      fill=Sfill,
      thick,
      rectangle,
      minimum size=6pt,
      inner sep=0pt
    ] at (\x,\y) {};
    \node[crossinglabel,text=Scolor] at (\x,\y) {$\lab$};
  }

  \node[above] at (0,6) {$A^-$};
  \node[above] at (5,6) {$A^+$};
  \node[below] at (0,-2) {$B^-$};
  \node[below] at (5,-2) {$B^+$};
\end{tikzpicture}
\caption{A permutation tile realizing the permutation \((1,2,0)\) as a first
return map. The crossings are labeled in their order along \(A\), so the
order along \(A\) is \((0,1,2,3,4,5,6,7,8,9,10,11)\), while the order along
\(B\) is \((0,7,8,9,6,1,2,5,10,11,4,3)\). The distinguished set
\(S=\{2,7,9\}\) has first-return \((7,9,2)\), which becomes \((1,2,0)\) after
relabeling \(S\) in its inherited order.}
\label{fig:permutation-tile}
\end{figure}

%% file: figures/corridor.tex
\begin{figure}[ht]
\centering
\begin{tikzpicture}[scale=.95, every node/.style={font=\scriptsize},
  conn/.style={very thick,->,>=stealth},
  bconn/.style={very thick,dashed,->,>=stealth},
  midconn/.style={very thick,postaction={decorate},
    decoration={markings,mark=at position .5 with {\arrow{stealth}}}},
  midbconn/.style={very thick,dashed,postaction={decorate},
    decoration={markings,mark=at position .5 with {\arrow{stealth}}}},
  plainconn/.style={very thick},
  plainbconn/.style={very thick,dashed}]

  \def\W{1.7}

  \foreach \x/\lab in {0/1,3/2,7/{q-1},10/q}{
    \draw[rounded corners] (\x,0) rectangle (\x+\W,1.6);
    \node at (\x+.85,.8) {$Q_{\lab}$};
    \fill (\x+.3,1.6) circle (1.7pt);
    \fill (\x+\W-.3,1.6) circle (1.7pt);
    \fill (\x+.3,0) circle (1.7pt);
    \fill (\x+\W-.3,0) circle (1.7pt);
  }
  \node at (5.7,.75) {$\cdots$};

  \node[font=\fontsize{4}{4.5}\selectfont] at (.3,1.35) {$A_1^-$};
  \node[font=\fontsize{4}{4.5}\selectfont] at (\W-.3,1.35) {$A_1^+$};
  \node[font=\fontsize{4}{4.5}\selectfont] at (.3,.25) {$B_1^-$};
  \node[font=\fontsize{4}{4.5}\selectfont] at (\W-.3,.25) {$B_1^+$};

  \node[font=\fontsize{4}{4.5}\selectfont] at (3.3,1.35) {$A_2^-$};
  \node[font=\fontsize{4}{4.5}\selectfont] at (3+\W-.3,1.35) {$A_2^+$};
  \node[font=\fontsize{4}{4.5}\selectfont] at (3.3,.25) {$B_2^-$};
  \node[font=\fontsize{4}{4.5}\selectfont] at (3+\W-.3,.25) {$B_2^+$};

  \node[font=\fontsize{4}{4.5}\selectfont] at (7.3,1.35) {$A_{q-1}^-$};
  \node[font=\fontsize{4}{4.5}\selectfont] at (7+\W-.3,1.35) {$A_{q-1}^+$};
  \node[font=\fontsize{4}{4.5}\selectfont] at (7.3,.25) {$B_{q-1}^-$};
  \node[font=\fontsize{4}{4.5}\selectfont] at (7+\W-.3,.25) {$B_{q-1}^+$};

  \node[font=\fontsize{4}{4.5}\selectfont] at (10.3,1.35) {$A_q^-$};
  \node[font=\fontsize{4}{4.5}\selectfont] at (10+\W-.3,1.35) {$A_q^+$};
  \node[font=\fontsize{4}{4.5}\selectfont] at (10.3,.25) {$B_q^-$};
  \node[font=\fontsize{4}{4.5}\selectfont] at (10+\W-.3,.25) {$B_q^+$};

  \draw[midconn]
    (\W-.3,1.6) -- (\W-.3,2.25) -- (3.3,2.25) -- (3.3,1.6);
  \draw[plainconn]
    (3+\W-.3,1.6) -- (3+\W-.3,2.25) -- (5.1,2.25);
  \draw[conn] (4.45,2.25) -- (4.75,2.25);
  \node at (5.7,2.25) {$\cdots$};
  \draw[plainconn]
    (6.3,2.25) -- (7.3,2.25) -- (7.3,1.6);
  \draw[conn] (6.65,2.25) -- (6.95,2.25);
  \draw[midconn]
    (7+\W-.3,1.6) -- (7+\W-.3,2.25) -- (10.3,2.25) -- (10.3,1.6);
  \draw[midconn]
    (10+\W-.3,1.6) -- (10+\W-.3,3.25) -- (.3,3.25) -- (.3,1.6);

  \draw[midbconn]
    (\W-.3,0) -- (\W-.3,-.65) -- (3.3,-.65) -- (3.3,0);
  \draw[plainbconn]
    (3+\W-.3,0) -- (3+\W-.3,-.65) -- (5.1,-.65);
  \draw[bconn] (4.45,-.65) -- (4.75,-.65);
  \node at (5.7,-.65) {$\cdots$};
  \draw[plainbconn]
    (6.3,-.65) -- (7.3,-.65) -- (7.3,0);
  \draw[bconn] (6.65,-.65) -- (6.95,-.65);
  \draw[midbconn]
    (7+\W-.3,0) -- (7+\W-.3,-.65) -- (10.3,-.65) -- (10.3,0);
  \draw[midbconn]
    (10+\W-.3,0) -- (10+\W-.3,-1.65) -- (.3,-1.65) -- (.3,0);

\end{tikzpicture}
\caption{The corridor closure.  The $A$-connectors are drawn in the upper
corridor and the $B$-connectors in the lower corridor, so the closure adds
no new $A$-$B$ intersections.}
\label{fig:corridor}
\end{figure}
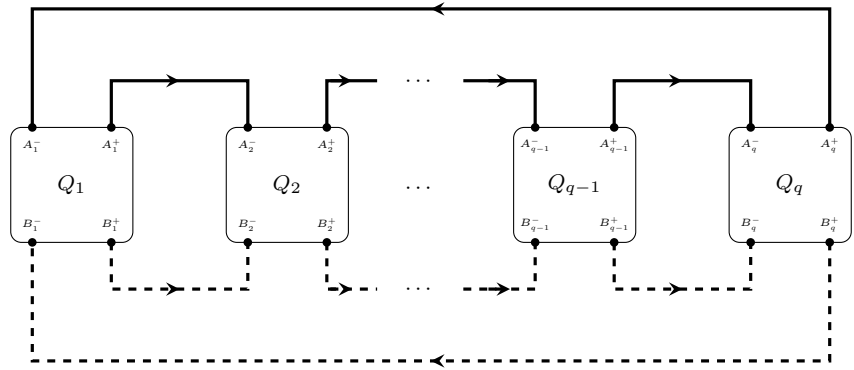

%% file: figures/scrambled.tex
\begin{figure}[ht]
\centering
\resizebox{.82\textwidth}{!}{%
\begin{tikzpicture}[x=1.0cm,y=1.0cm, every node/.style={font=\scriptsize}]
  \colorlet{Vcolor}{black}
  \colorlet{HColor}{red!75!black}
  \colorlet{BoxFill}{red!10}
  \colorlet{GlobalColor}{cyan!75!black}

  \tikzset{
    local dotted arc/.style={
      line width=1.35pt,
      densely dotted,
      draw=HColor,
      postaction={decorate},
      decoration={
        markings,
        mark=at position 0.5 with {\arrow{stealth}}
      }
    },
    global dotted arc/.style={
      line width=1.15pt,
      densely dotted,
      draw=GlobalColor,
      postaction={decorate},
      decoration={
        markings,
        mark=at position 0.5 with {\arrow{stealth}}
      }
    },
    block box/.style={
      rounded corners=3pt,
      draw=HColor,
      fill=BoxFill,
      thick
    }
  }

  \path[use as bounding box] (-1.15,-3.65) rectangle (10.15,2.35);

  \draw[block box] (-.35,-1.45) rectangle (4.35,1.35);
  \draw[block box] (4.65,-.85) rectangle (6.35,.85);
  \draw[block box] (6.65,-1.45) rectangle (9.35,1.35);

  \node[text=HColor] at (2.0,1.55) {$H_1$};
  \node[text=HColor] at (5.5,1.20) {$H_2$};
  \node[text=HColor] at (8.0,1.55) {$H_3$};

  \draw[GlobalColor,line width=1pt]
    (-.85,0) -- (9.85,0) -- (9.85,1.85) -- (-.85,1.85) -- cycle;

  \draw[HColor,line width=1.15pt] (0,0) -- (4,0);
  \draw[HColor,line width=1.15pt] (5,0) -- (6,0);
  \draw[HColor,line width=1.15pt] (7,0) -- (9,0);

  \draw[GlobalColor,line width=1pt,->,>=stealth] (5.3,1.85) -- (4.7,1.85);


  \draw[local dotted arc] (0,0) arc[start angle=180,end angle=360,radius=.5]; 
  \draw[local dotted arc] (1,0) arc[start angle=180,end angle=0,radius=1];    
  \draw[local dotted arc] (3,0) arc[start angle=180,end angle=360,radius=.5]; 
  \draw[local dotted arc] (2,0) arc[start angle=0,end angle=180,radius=1];    

  \draw[global dotted arc] (4,0) arc[start angle=180,end angle=0,radius=1];   

  \draw[local dotted arc] (6,0) arc[start angle=360,end angle=180,radius=.5]; 

  \draw[global dotted arc] (5,0) arc[start angle=180,end angle=0,radius=1];   

  \draw[local dotted arc] (7,0) arc[start angle=180,end angle=360,radius=1];  
  \draw[local dotted arc] (9,0) arc[start angle=0,end angle=180,radius=.5];   

  \draw[global dotted arc] (8,0) arc[start angle=360,end angle=180,radius=3]; 

  \foreach \x in {0,1,...,9}{
    \node[fill=Vcolor,circle,minimum size=5pt,inner sep=0pt] at (\x,0) {};
    \node[below=7pt,text=Vcolor] at (\x,0) {$\x$};
  }
\end{tikzpicture}%
}
\caption{A drawing of the multigraphs $H_j(\sigma)$ for
$\sigma=(2,0,1,3,4,6,5,7,9,8)$, with
$I_1=\{0,1,2,3,4\}$, $I_2=\{5,6\}$, and
$I_3=\{7,8,9\}$.}
\label{fig:scrambled}
\end{figure}
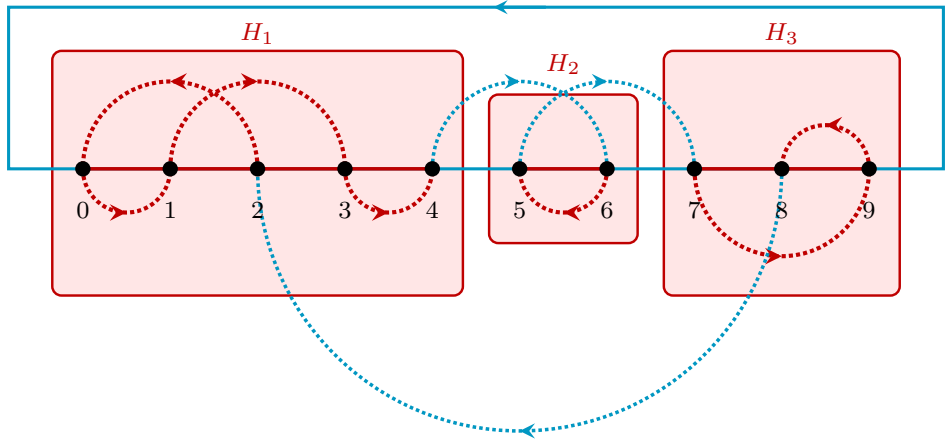

%% file: refs.bib
@misc{georgiev2025mathematical,
  title={Mathematical exploration and discovery at scale},
  author={Georgiev, Bogdan and G{\'o}mez-Serrano, Javier and Tao, Terence and Wagner, Adam Zsolt},
  journal={arXiv preprint arXiv:2511.02864},
  year={2025}
}

@article{LiptonTarjan,
  title={A separator theorem for planar graphs},
  author={Lipton, Richard J and Tarjan, Robert Endre},
  journal={SIAM Journal on Applied Mathematics},
  volume={36},
  number={2},
  pages={177--189},
  year={1979},
  publisher={SIAM}
}

@book{leighton1983complexity,
  title={Complexity issues in VLSI: optimal layouts for the shuffle-exchange graph and other networks},
  author={Leighton, Frank Thomson},
  year={1983},
  publisher={MIT press}
}

@inproceedings{djidjev2001improved,
  title={An improved lower bound for crossing numbers},
  author={Djidjev, Hristo and Vrt’o, Imrich},
  booktitle={International Symposium on Graph Drawing},
  pages={96--101},
  year={2001},
  organization={Springer}
}

@inproceedings{pach1994applications,
  title={Applications of the crossing number},
  author={Pach, J{\'a}nos and Shahrokhi, Farhad and Szegedy, Mario},
  booktitle={Proceedings of the tenth annual symposium on Computational geometry},
  pages={198--202},
  year={1994}
}

@article{djidjev2003crossing,
  title={Crossing numbers and cutwidths},
  author={Djidjev, Hristo and Vrt'o, Imrich},
  journal={Journal of Graph Algorithms and Applications},
  volume={7},
  number={3},
  pages={245--251},
  year={2003}
}

@misc{DBLP:journals/corr/abs-2506-13131,
  author       = {Alexander Novikov and
                  Ng{\^{a}}n Vu and
                  Marvin Eisenberger and
                  Emilien Dupont and
                  Po{-}Sen Huang and
                  Adam Zsolt Wagner and
                  Sergey Shirobokov and
                  Borislav Kozlovskii and
                  Francisco J. R. Ruiz and
                  Abbas Mehrabian and
                  M. Pawan Kumar and
                  Abigail See and
                  Swarat Chaudhuri and
                  George Holland and
                  Alex Davies and
                  Sebastian Nowozin and
                  Pushmeet Kohli and
                  Matej Balog},
  title        = {AlphaEvolve: {A} coding agent for scientific and algorithmic discovery},
  journal      = {CoRR},
  volume       = {abs/2506.13131},
  year         = {2025},
  url          = {https://doi.org/10.48550/arXiv.2506.13131},
  doi          = {10.48550/ARXIV.2506.13131},
  eprinttype   = {arXiv},
  eprint       = {2506.13131},
  bibsource    = {dblp computer science bibliography, https://dblp.org}
}

@article{Poincare1912Sur,
	author = {Poincar{\' e}, H.},
	journal = {Rendiconti del Circolo Matematico di Palermo},
	number = {1},
	year = {1912},
	month = {12},
	pages = {375--407},
	publisher = {{Springer Science and Business Media LLC}},
	title = {Sur un th{\' e}or{\` e}me de g{\' e}om{\' e}trie},
	volume = {33},
}

@article{DIFRANCESCO199797,
title = {Meander, folding, and arch statistics},
journal = {Mathematical and Computer Modelling},
volume = {26},
number = {8},
pages = {97-147},
year = {1997},
issn = {0895-7177},
doi = {https://doi.org/10.1016/S0895-7177(97)00202-1},
url = {https://www.sciencedirect.com/science/article/pii/S0895717797002021},
author = {P. {Di Francesco} and O. Golinelli and E. Guitter}
}

@misc{richtpss,
  author       = {Schwartz, Richard},
  title        = {The {T}opological {S}alesman {P}roblem},
  year         = {2025},
  howpublished = {Preprint},
  note         = {Available at \url{https://www.math.brown.edu/reschwar/Papers/salesman.pdf}},
  eprinttype   = {Self-Published}
}

@article{FurediKundgen2006,
  author  = {Zolt{\'a}n F{\"u}redi and Andr{\'e} K{\"u}ndgen},
  title   = {Moments of graphs in monotone families},
  journal = {Journal of Graph Theory},
  volume  = {51},
  number  = {1},
  pages   = {37--48},
  year    = {2006},
  doi     = {10.1002/jgt.20119}
}

@book{cover2006elements,
  author    = {Cover, Thomas M. and Thomas, Joy A.},
  title     = {Elements of Information Theory},
  publisher = {Wiley-Interscience},
  year      = {2006},
  edition   = {2},
  address   = {Hoboken, NJ},
  isbn      = {978-0-471-24195-9},
  doi       = {10.1002/047174882X}
}
